\documentclass[11pt,final]{amsart}
\usepackage{amsmath,amssymb,amsthm,amsfonts,mathrsfs,amsopn}
\usepackage[all]{xy}
\usepackage{dsfont}
\usepackage{hyperref}
\usepackage{color}
\usepackage{esint}
\usepackage{mathtools}
\mathtoolsset{showonlyrefs}
\usepackage{slashed}

\usepackage{tikz-cd}

\usepackage{faktor}

\newcommand{\tnorm}[1]{%
  \left\vert\kern-0.3ex\left\vert\kern-0.3ex\left\vert #1 \right\vert\kern-0.3ex\right\vert\kern-0.3ex\right\vert
}

\usepackage[margin=0.9in]{geometry}

\newcommand{\Abracket}[1]{\left<#1\right>} 
\newcommand{\parenthesis}[1]{\left(#1\right)} 
\newcommand{\braces}[1]{\left\{#1\right\}} 

\newcommand{\R}{\mathbb{R}}
\newcommand{\C}{\mathbb{C}}
\newcommand{\T}{\mathbb{T}}
\newcommand{\bZ}{\mathbb{Z}}

\newcommand{\dd}{\mathop{}\!\mathrm{d}}
\newcommand{\eps}{\varepsilon}
\newcommand{\p}{\partial}
\newcommand{\sph}{\mathbb{S}}

\newcommand{\D}{\slashed{D}}
\newcommand{\pd}{\slashed{\partial}}
\newcommand{\pD}{\mathbf{D}}

\DeclareMathOperator{\Dom}{Dom}
\DeclareMathOperator{\dv}{\dd{vol}}

\DeclareMathOperator{\dist}{dist}
\DeclareMathOperator{\Eigen}{Eigen}

\DeclareMathOperator{\id}{Id}

\DeclareMathOperator{\Mtp}{Multi} 

\DeclareMathOperator{\proj}{Pr}

\DeclareMathOperator{\Span}{Span}

\DeclareMathOperator{\Spect}{Spect}

\DeclareMathOperator{\spt}{spt}

\DeclareMathOperator{\vol}{vol}

\DeclareMathOperator{\up}{up}
\DeclareMathOperator{\low}{down}

\newtheorem{thm}{Theorem}[section]

\newtheorem{lemma}[thm]{Lemma}
\newtheorem{prop}[thm]{Proposition}
\newtheorem{cor}[thm]{Corollary}

\newtheorem{rmk}[thm]{Remark}
\newtheorem{eg}{Example}

\title[Stationary periodic solutions]{Stationary periodic solutions to nonlinear Dirac equations with non-coercive potentials}


\thanks{The authors are supported by Beijing Natural Science Foundation No. 1262018.}

\author[R. Wu]{Ruijun Wu}
\address{Ruijun Wu, School of Mathematics and Statistics, Beijing Institute of Technology, Zhongguancun South Street No. 5, 100081 Beijing, P.R. China.}
\email{ruijun.wu@bit.edu.cn}

\author[F. Zhang]{Fuping Zhang}
\address{Fuping Zhang, School of Mathematics and Statistics, Beijing Institute of Technology, Zhongguancun South Street No. 5, 100081 Beijing, P.R. China.}
\email{fuping.zhang@bit.edu.cn}

\begin{document}

\begin{abstract}
    Motivated by the Soler model, we study a nonlinear Dirac equation with a non-coercive Soler-type nonlinearity. Stationary periodic solutions are obtained via variational methods. This amounts to studying the equation on a three-dimensional torus, where the spectrum of the physical Dirac operator on the torus needs to be clarified. To overcome the failure of coercivity of the potential we use a coercive perturbation. 
    By carefully analyzing the critical levels and the linking structures, uniform estimates for the perturbed critical points are obtained when the frequency is sufficiently close to the mass.
    This allows us to pass to the zero-perturbation limit and obtain a nontrivial periodic solution. 
    The argument also extends to suitable scalar external fields.
\end{abstract}


\maketitle

\section{Introduction}

In physical models for fermionic fields, the equations of motion often appear as Dirac equations with various nonlinearities.
In this work we consider periodic solutions on~$\R^{1+3}$, which are actually spatially periodic solutions and stationary in the physical sense, i.e., they propagate via a complex phase factor without changing shape. 
Here we consider nonlinearities which are of the Soler type; namely, the nonlinearity is a function of a Lorentz scalar quantity such as~$\bar{\psi}\psi$, instead of an Ambrosetti-Rabinowitz-type where the nonlinearity is a function of the Euclidean amplitude~$|\psi|$. 
Since such solutions fail to decay at infinity, we have to employ a perturbative variational scheme to obtain a nontrivial solution.

More precisely, we will be concerned with the following nonlinear Dirac equation
\begin{align}\label{eq:4D nonlinear Dirac}
    \sum_{\mu=0}^{3} i\gamma^\mu\p_\mu\Psi-m \Psi + \gamma^0 \p_\psi F(\Psi)=0
\end{align}
where~$\Psi\colon \R^{1+3}\to \C^4$ stands for a spinor field,~$m>0$ corresponds to the mass of the particles, and~$F\colon \R^{1+3}\times\C^4\to \R$ is the self-coupling of particles in the system, which is chosen according to the complexity of the physical models, see e.g.~\cite{BartschDing2006Solutions,Ding2007variational, Esteban2002overview,Psarelli2005Maxwell,Soler1970Classical, Thaller1992Dirac} and the references therein.
The~$\gamma^\mu$'s are given as follows.
The matrix~$\gamma^0$ takes the form
\begin{align}
    \gamma^0= \begin{pmatrix} I_2 & 0 \\ 0 & -I_2\end{pmatrix}
    =\begin{pmatrix}
        1 & & &  \\ & 1 & & \\ & & -1 & \\ & &  & -1
    \end{pmatrix}
\end{align}
and induces a decomposition of the spinors into~$\pm1$ eigenspaces of~$\gamma^0$.
As for the others, we denote by
\begin{align}
    {\sigma^1} = \left( {\begin{array}{*{20}{c}}
			{0}&1 \\
			1&{0}
	\end{array}} \right),\;{\sigma^2} = \left( {\begin{array}{*{20}{c}}
			{0}&{ - i} \\
			i&{0}
	\end{array}} \right),\;{\sigma^3} = \left( {\begin{array}{*{20}{c}}
			1&{0} \\
			{0}&{ - 1}
	\end{array}} \right)
\end{align}
for the Pauli matrices, and the gamma matrices are
\begin{align}
    {\gamma^k} = \left( {\begin{array}{*{20}{c}}
			{0}&{{\sigma^k}} \\
			{{-\sigma^k}}&{0}
	\end{array}} \right),\quad\text{for}\quad k= 1, 2, 3.
\end{align}
Note that they satisfy the following Clifford relations: for~$j,k\in\braces{1,2,3}$,
\begin{align}
    \sigma^j\sigma^k+\sigma^k\sigma^j= 2\delta^{jk}I_2, & &
    \gamma^j\gamma^k + \gamma^k\gamma^j= -2\delta^{jk}I_4.
\end{align}
Moreover,
\begin{align}
    \gamma^0\gamma^k + \gamma^k \gamma^0 =0, & &
    (\gamma^0)^2=I_4.
\end{align}
Then~$\sum_{\mu} i\gamma^\mu \p_\mu$ is the four-dimensional space-time Dirac operator.
With the time direction involved, this operator is not elliptic.

A typical ansatz for the solutions is the so-called \emph{stationary solutions} and takes the form
\begin{align}
    \Psi(t,\vec{x})= e^{-i a t}\psi(\vec{x}),
\end{align}
for some~$a\in\R$,~$a\neq 0$.
Such solutions are regarded as “particle-like solutions” and hence attract much attention. 
For spinors of such form and for~$F$ invariant under phase rotations of spinor and translation of time ( i.e., ~$F((t,\vec{x}),e^{i\theta}\Psi)=F(\vec{x},\Psi)$),~\eqref{eq:4D nonlinear Dirac} becomes
\begin{align}
    i\gamma^k\p_k\psi + a\gamma^0\psi-m\psi +\gamma^0\p_\psi F(\vec{x}, \psi)=0.
\end{align}
Equivalently, we are led to consider the equation
\begin{align}\label{eq:NDE-spatial}\tag{*}
    -i\gamma^0 \gamma^k\p_k \psi + m\gamma^0\psi -a\psi-\p_\psi F(\vec{x},\psi)=0.
\end{align}
Equations of this form have been studied extensively in the literature, see e.g.~\cite{BalabaneCazenave1988Existence, Balabane1988Existence, Balabane1990Existence, Cazenave1986Existence,EstebanSere1995stationary,Merle1988Existence} and the references quoted therein.
Different from~\cite{EstebanSere1995stationary}, here we are concerned with periodic solutions, namely solutions satisfying
\begin{align}
    \psi(\vec{x}+\tau)=\psi(\vec{x}), \quad \forall \tau\in \Gamma
\end{align}
where~$\Gamma\subset \R^3$ is a 3-dimensional lattice.
They are referred to as the standing periodic waves for~\eqref{eq:4D nonlinear Dirac}. 

\ 

In quantum field theory the Hamiltonian is characterized by symmetry of the system, and the periodicity is thus frequently imposed. 
In one-dimensional dynamics this is usually the case. 
According to the features of the particle system, it is often required that the nonlinearity has spatial periodicity. 
Here we are in a setting similar to that of~\cite{DingLiu2014periodic}, namely we seek periodic solutions for the nonlinear Dirac equation~\eqref{eq:NDE-spatial}. 
The difference from~\cite{DingLiu2014periodic} lies in the nonlinearity, which will become clear later. 

Geometrically, this amounts to considering the equation~\eqref{eq:NDE-spatial} on the compact manifold~$\mathbb{T}^3=\faktor{\R^3}{\Gamma}$.
The spectral properties for the linearized operators on torus are in great contrast to those on~$\R^3$, so is the analysis needed for the variational scheme. 
For example, we cannot appeal to the Fourier space techniques to obtain enough linking structures, and we do not have a global Pohozaev identity to control the kinetic energy and the potential energy separately, which are crucial for the proof in~\cite{EstebanSere1995stationary} to obtain solutions on~$\R^3$. 
On the other hand, the torus is still good enough to allow for variational analysis and we can change local strategies to carry out the variational method in~\cite{EstebanSere1995stationary} to obtain nontrivial solutions. 

\ 

The main motivation of the nonlinearities comes from the following examples. 
In the Soler model~\cite{Soler1970Classical}, the self-coupling has the form 
\begin{align}\label{eq:F Soler}
    F(\psi)
    =\frac{1}{2}G(\bar{\psi}\psi)
\end{align}
with~$G\in C^2(\R,\R)$ and~$G(0)=0$, and~$\bar{\psi}\psi\equiv \Abracket{\gamma^0\psi,\psi}$.
To solve the corresponding nonlinear Dirac equation on~$\R^3$, a suitable radial ansatz was used to reduce the equation to ODEs and then the shooting method can provide some solutions. 
These methods cannot be applied in the periodic setting because of the lack of a good radial ansatz. 
A more general form of~$F$ arises in~\cite{Finkelstein1951Non-linea} in which the authors studied the symmetric coupling between nucleons, muons, and leptons, with 
\begin{align}\label{eq:nonlinearity Finkelstein}
    F(\psi) = \frac{1}{2}|\bar{\psi}\psi|^2 + b|\bar{\psi}\gamma^5\psi|^2
\end{align}
where~$\gamma^5= i\gamma^0\gamma^1\gamma^2\gamma^3$.
In such models, \emph{the existence theory still remains largely open}. 
Indeed, the nonlinearity in~\eqref{eq:nonlinearity Finkelstein} has quartic growth which is supercritical for the relevant Sobolev embedding and is therefore not suitable for the variational scheme used here.
Since standard fermionic models naturally involve Soler-type scalar couplings rather than nonlinearities depending only on~$|\psi|$, it is natural to study nonlinear Dirac equations of this type.
This is the main motivation for our present work.

\

In this work we only deal with the subcritical case. 
To be precise, we impose the following hypotheses on the nonlinearity~$F$.
Let~$F\colon \mathbb{T}^3\times \C^4\to \R_+$ be a function defined on the trivial spinor bundle of~$\mathbb{T}^3$. 
Suppose that there exist constants~$A_j>0$,~$1\leq j\leq 5$, and~$\nu>1$,~$2<\alpha_1\leq\alpha_2<3$,~$\alpha\in (0,1)$,~$\beta>3$ such that
\begin{itemize}
    \item[(F1)] $ 0 \leq F(x,\psi) \leq A_1 \big( |\psi|^{\alpha_1} + |\psi|^{\alpha_2} \big)$, $\forall x\in\mathbb{T}^3$ and $\forall \psi \in \mathbb{C}^4 $;

    \item[(F2)] \( F \in C^{1,\alpha}_{loc} \), $F(x,0) =0 $, $\p_\psi F(x,0)=0$, and
          \( |\p_{\psi} F(x,\psi)| \leq A_2 |\psi|^{\alpha_2 - 1} \) for \( |\psi| \) large;

    \item[(F3)] \(  \p_\psi F(x,\psi)[\psi]\geq \alpha_1 F(x,\psi) \);

    \item[(F4)] \( F(x,\psi) \geq A_3 |\bar{\psi}\psi|^\nu - A_4 \);

    \item[(F5)] for any~$\delta > 0$, there is a~$C_\delta > 0$, such that $\forall \psi \in \mathbb{C}^4$,~$\forall x\in\mathbb{T}^3$,
        \begin{align}
        |\p_\psi F(x,\psi)| \leq A_5 \left( \delta + C_\delta F(x,\psi)^{\frac{1}{\beta}} \right) |\psi|.
        \end{align}
\end{itemize}
Note that by (F1) and (F4) we necessarily have~$\nu\leq \frac{\alpha_2}{2}<\frac{3}{2}$.

\begin{eg}
    We discuss some examples before proceeding further. 
    
    The original Soler nonlinearity~\cite{Soler1970Classical} takes the form
    \begin{align}
        F_{Soler}(\psi)=(\bar{\psi}\psi)^2
    \end{align}
    which has quartic growth. 
    However, since we will work with spinors in the class~$H^{\frac{1}{2}}$ (see Section~\ref{sect:prelim}), which are only~$L^3$ but not necessarily~$L^4$, we cannot handle the Soler nonlinearity above. 

    To match the above hypotheses, we consider 
    \begin{align}
        F_{1}(x,\psi)=h(x) |\bar{\psi}\psi|^{\frac{5}{4}}
    \end{align}
    where~$h(x)\in C^\infty(\mathbb{T}^3,\R)$ is a bounded strictly positive smooth function.
    It satisfies the above hypotheses with~$\nu=\frac{5}{4}$,~$\alpha_1=\alpha_2=\frac{5}{2}$, and~$\beta\in (3,5]$. 
    Note that~$F(x,\psi)$ is only~$C^1$ but not~$C^2$, although~$\bar{\psi}\psi$ appears to be a quadratic term. 
    Indeed, let~$\psi_0 \neq 0$ be such that~$\bar{\psi}_0\psi_0=0$, for instance~$\psi_0=(1,0,1,0)^T$, and let~$\phi_0=\gamma^0\psi_0=(1,0,-1,0)^T$.
    Then along the line~$\psi(t)=\psi_0+ t\phi_0$, we have~$\bar{\psi}(t)\psi(t)= 4t$, hence
    \begin{align}
        F_1(\psi(t))=|\bar{\psi}(t)\psi(t)|^{\frac{5}{4}}=|4t|^{\frac{5}{4}}
    \end{align}
    which is not~$C^2$ at~$t=0$. 
    Thus~$|\bar{\psi}\psi|^{\frac{5}{4}}$ is not~$C^2$ along the Lorentz null cone
    \begin{align}
        \braces{\psi\mid \bar{\psi}\psi=0, \; \psi\neq 0}. 
    \end{align}
    Nevertheless it is~$C^{1,\alpha}$ with~$\alpha=\frac{1}{4}$. 
    This explains why we cannot impose~$C^2$ regularity as long as we aim to deal with pure Soler type nonlinearities. 

    More generally, we can consider nonlinearities of the form 
    \begin{align}
    F_2(x,\psi)= \sum_{j=1}^k h_j(x)|\bar{\psi}\psi|^{n_j} ,
    \end{align}
    where~$h_j$'s are smooth and satisfy~$0<\frac{1}{C}\leq h_j(x)\leq C<+\infty$, and~$1<n_1 < n_2 <\cdots < n_k<\frac{3}{2}$. 
    In this case we can take
    \begin{align}
        0<\alpha<n_1-1, & & 3<\beta<\frac{n_k}{n_k-1}. 
    \end{align}

    We should point out that the above hypotheses are automatically satisfied by nonlinear functions in the Euclidean amplitude~$|\psi|$, for instance
    \begin{align}
        F_3(x,\psi)= \sum_{j=1}^k h_j(x)|\psi|^{2n_j} ,
    \end{align}
    with~$h_j$'s and~$n_j$'s as above. 
    Nonlinearities of this type have been studied in~\cite{DingLiu2014periodic}. 
    However, their methods cannot be used to deal with Soler-type nonlinearities, due to the non-coercivity of the Lorentz scalar~$\bar{\psi}\psi$. 
    We will comment on this later. 

    Indeed, our hypotheses allow for nonlinearities of more complicated form. 
    For instance, 
    \begin{align}
        F_4(x,\psi)
        = h_1(x) |\bar{\psi}\psi|^{\frac{5}{4}} + h_2(x) |\psi|^{\frac{1}{2}}\Abracket{\gamma^6\psi,\psi}
    \end{align}
    where~$h_1,h_2$ are again positive smooth functions and~$\gamma^6$ is a nonnegative definite but degenerate matrix, say 
    \begin{align}
        \gamma^6=\begin{pmatrix}
            1 & -1 & 0 & 0 \\ -1 & 1 & 0 & 0 \\ 0 & 0 & 1 & 0 \\  0& 0& 0& 1 
        \end{pmatrix}
    \end{align}
\end{eg}


The main difficulty in the present problem comes from the non-coercivity of the Lorentz scalar~$\bar{\psi}\psi$, which prevents us from applying directly the variational arguments developed for nonlinearities depending on the Euclidean amplitude~$|\psi|$. 
Although our perturbative strategy is inspired by~\cite{EstebanSere1995stationary}, the periodic setting requires a different compactness argument. 
The key point is to establish a uniform~$L^3$ estimate for the perturbed critical points. 
If such a bound failed, a blow-up normalization together with (F4) would force the limiting spinor into the Lorentz null cone. 
On the other hand, the spectral structure of the Dirac operator near the eigenvalue~$m$, combined with the projection estimate and assumption~(F5), rules out such a nontrivial limit when~$a$ is sufficiently close to~$m$. 
This provides the compactness needed to let the perturbation parameter tend to zero and allows us to treat a class of non-coercive Soler-type nonlinearities which is not covered by the existing periodic variational theory.

We can now state our main result.
\begin{thm}\label{thm:main}
    Assume that (F1-5) hold true. 
    There exists~$m_*\in (0,m)$ such that for each~$a\in (m_*,m)$, there exists a nonzero ~$C^1$ periodic solution of~\eqref{eq:NDE-spatial}.
\end{thm}
Unlike the result in~\cite{EstebanSere1995stationary}, we do not obtain existence of nonzero solutions for each~$a\in (0,m)$.
This is because we cannot get uniform estimates for perturbed solutions for small~$a$.
A compactness-at-infinity type argument implies the uniform~$L^3$ estimate for~$a$ close to~$m$. 
We refer to Sections~\ref{sect:uniform estimate} and~\ref{sect:last} for more technical details. 

We also remark that, if~$F$ depends only on~$\psi\in\C^4$ and does not depend on~$x\in\mathbb{T}^3$ explicitly, then one can search for constant solutions of 
\begin{align}
    m\gamma^0\psi -a \psi = \dd F(\psi). 
\end{align}
This can in turn be obtained by looking at some function defined on~$\C^4\cong \R^8$. 
However, if~$F$ depends genuinely on~$x\in\mathbb{T}^3$, then constant solutions are in general not expected.

We outline the proof here. 
The equation~\eqref{eq:NDE-spatial} admits a variational structure, namely it is the Euler--Lagrange equation of an action functional~$J$, see Section~\ref{sect:variational str}. 
Unfortunately this functional~$J$ does not satisfy the Palais-Smale condition in general. 
As a remedy we introduce a coercive perturbation with a small parameter~$\eps\in (0,1]$ in Section~\ref{sect:perturbation}. 
Then we show that the perturbed functionals meet the Palais-Smale condition, although the norms of the~$(PS)_c$ sequences are bounded by a quantity depending on~$\frac{1}{\eps}$. 
The spectral information yields a linking geometry whose associated level estimates are uniform in the perturbation parameter~$\eps$. 
Since the gradients of the perturbed functionals fail to be locally Lipschitz, we take suitable approximating vector fields which are bounded and locally Lipschitz to generate flow-type deformations. 
Then we use such deformations to define the min-max levels and show that they are critical levels.
These are done in Section~\ref{sect:perturbation}. 
By carefully estimating the critical levels and using the structure of~$F$, in particular (F5), we can get uniform estimates for the perturbed solutions as long as~$a$ is close to~$m$, see Section~\ref{sect:uniform estimate}. 
This allows us to pass along a subsequence to a nontrivial limit which is the desired nontrivial periodic solution, see Section~\ref{sect:last}.

\

When there is an external field, the following equation is also considered in~\cite{DingRuf2008Solutions}:
\begin{align}\label{eq:NLD with external field}
    -i\gamma^0\gamma^k\p_k\psi + m\gamma^0\psi -a\psi - M\psi -\p_\psi F(x, \psi)=0,
\end{align}
where~$M\colon \mathbb{T}^3\to\R$ stands for the external field, which can be either matrix-valued or scalar-valued.
Here we restrict attention to a scalar-valued external field~$M\in C^2(\mathbb{T}^3,\R)$ which does not oscillate much.
Then the argument carries over to this case, and we also obtain a nontrivial periodic~$C^1$ solution. 

\begin{prop}\label{prop:external}
    Assume (F1-5).  
    Then there exists~$m_{*}\in (0,m)$ such that if~$a\in (0,m)$ and~$M\in C^2(\mathbb{T}^3,\R)$ satisfies 
    \begin{align}
        m_{*} < a + M(x) < m, \qquad \forall x\in\mathbb{T}^3,
    \end{align}
    then the equation~\eqref{eq:NLD with external field} admits a nonzero~$C^1$ periodic solution. 
\end{prop}
Note that this differs from the results in~\cite{DingRuf2008Solutions} not only in the non-coerciveness of the nonlinearity, but also in the growth order of the nonlinearity. 
They dealt essentially with the quadratic growth case, and we are concerned with the super-quadratic case, hence the results are not overlapping. 
Furthermore, multiplicity results will be addressed in a future work; see Remark~\ref{rmk:multiplicity}.

\

\noindent{\bf Acknowledgments.} We would like to thank Nadine Grosse, Long Meng and Zhenyao Sun for helpful conversations on these problems.


\section{Preliminaries} \label{sect:prelim}

In this section we collect some basic facts about the Dirac operators and the working space~$H^{\frac{1}{2}}(\mathbb{T}^3,\C^4)$.

\subsection{The Dirac operators}

We will denote
\begin{align}
    \D= \sum_{k=1}^3 -i\gamma^0\gamma^k\p_k, & &
    \mbox{and} & &
    \pD=\sum_{k=1}^3 -i\gamma^0\gamma^k\p_k+m\gamma^0 .
\end{align}
Note that the coefficients of~$\D$ satisfy
\begin{align}
    (-i\gamma^0\gamma^k)(-i\gamma^0\gamma^j)+ (-i\gamma^0\gamma^j)(-i\gamma^0\gamma^k)=-2\delta^{jk} I_4, \qquad \forall j,k\in \braces{1,2,3}.
\end{align}
Thus~$\D$ satisfies the geometers' convention.
Actually, the relation between the gamma matrices and the Pauli matrices already indicates that~$\D$ is intimately related to the intrinsic Dirac operator of~$\R^3$ or~$\mathbb{T}^3$.
In matrix form, we have
\begin{align}
    \D =
    \begin{pmatrix}
        0 & -i\mathbf{\sigma}\cdot\nabla \\ -i\mathbf{\sigma}\cdot\nabla & 0
    \end{pmatrix}
    = \begin{pmatrix}
        0 & \pd \\ \pd & 0
    \end{pmatrix}
\end{align}
where
\begin{align}
    \pd\equiv -i\mathbf{\sigma}\cdot \nabla = \sum_{k=1}^3 -i\sigma^k\p_k
\end{align}
denotes the intrinsic Dirac operator on~$\R^3$ and on~$\mathbb{T}^3$ (equipped 
with the trivial spin structure) acting on~$\C^2$-valued functions which are precisely the spinors on~$\R^3$ or~$\mathbb{T}^3$.
This agrees with the fact that the spinor bundle on~$\R^{1+3}$ restricts on~$\R^3$ to a double of the spinor bundles of~$\R^3$ which is inherited by the quotient~$\mathbb{T}^3$.
If multiplied by the timelike normal vector field (by Clifford multiplication), we get a double of the intrinsic Dirac operator, namely
\begin{align}
    \nu\cdot \D=
    \begin{pmatrix}
        0 & I_2 \\ I_2 & 0
    \end{pmatrix} \D
    =\begin{pmatrix}
        \pd & 0 \\ 0 & \pd
    \end{pmatrix}.
\end{align}
Thus we will temporarily call~$\D$ the \textbf{geometric Dirac operator}.
Note that the constant functions are parallel spinors for this spinor bundle, which are also in the kernel of~$\D$.

\begin{rmk}
    For nontrivial spin structures, the spinor bundle is also nontrivial and the constant functions are no longer sections of the spinor bundle; indeed there are no harmonic spinors for the nontrivial spin structure.

    The presence of the kernel of~$\D$, which coincides with the constant sections, is crucial for our construction of the variational scheme in this work.
    The case of nontrivial spin structures will be the subject of a future work.
\end{rmk}

In physical models of massive particles, the operator~$\pD=\D+m\gamma^0$ is frequently used to define the Lagrangian in field theory, see~\cite{Thaller1992Dirac}.
When the space-time is~$\R^{1+3}$, the mass~$m>0$ provides a gap in the spectrum of the Dirac operator~$\pD$ which is essential for the variational theory and for applications in QFT and quantum mechanics~\cite{EstebanSere1995stationary,Esteban2002overview,Sere2023new}.
This fact will also be used in our treatment of the equation~\eqref{eq:NDE-spatial}, while one should note that the geometric Dirac operator~$\D$ has~$0$ as an eigenvalue.
Note that~$\pD$ is the operator originally found by P. Dirac in 1928~\cite{Dirac1928The1,Dirac1928The2, Dirac1931Quantised} as a square root of the Klein-Gordon operator, which also contains a mass term.
For these reasons we will call~$\pD$ the \textbf{physical Dirac operator}.

\subsection{Spectrum of the Dirac operators}

The spectrum of the intrinsic Dirac operator~$\pd$ on~$\mathbb{T}^3$ is explicitly known; see e.g.~\cite{Friedrich1984zur} and~\cite[Chapter 2.1]{Ginoux2009Dirac}. We briefly recall it here. 

Let~$\Gamma\subset\R^3$ be a 3-dimensional lattice with dual lattice~$\Gamma^*$, namely
\begin{align}
    \Gamma^*=\braces{\zeta^*\in (\R^3)^*\mid \zeta^*(\Gamma)\subset\mathbb{Z}}.
\end{align}
For simplicity we may consider
\begin{align}
    \Gamma=\ell_1\bZ \times \ell_2\bZ \times \ell_3\bZ, & &
    \mbox{ with } & &
    \Gamma^*= \frac{1}{\ell_1}\bZ\times \frac{1}{\ell_2}\bZ \times \frac{1}{\ell_3}\bZ.
\end{align}
Without loss of generality, we assume~$0<\ell_1\leq \ell_2 \leq \ell_3 <+\infty$.
For this lattice, the corresponding torus is
\begin{align}
    \mathbb{T}^3=\mathbb{T}^3_\Gamma = \faktor{\R^3}{\Gamma}=\sph^1(\frac{\ell_1}{2\pi}) \times \sph^1(\frac{\ell_2}{2\pi}) \times \sph^1(\frac{\ell_3}{2\pi}),
\end{align}
where~$\sph^1(r)$ denotes the standard circle with radius~$r>0$.
The global periodic coordinates of~$\mathbb{T}^3_{\Gamma}$ are~$(\theta^1,\theta^2,\theta^3)$.
Then
\begin{align}
    \int_{\mathbb{T}^3} \dv = \int_{\mathbb{T}^3} \dd\theta^1\dd\theta^2\dd\theta^3= \ell_1 \ell_2 \ell_3.
\end{align}
We remark that these parameters~$\ell_1,\ell_2,\ell_3$ are not essentially needed in this work; we only want to indicate the difference between~$\Gamma$ and its dual lattice~$\Gamma^*$.
However, they will play a role in future work, where we need to take one of them, sometimes all of them, to be relatively small.

It is well-known that
\begin{align}
    \Spect(\pd_{\mathbb{T}^3})=\braces{\pm2\pi |\zeta^*| \; \mid\;  \zeta^*\in \Gamma^* }.
\end{align}
Furthermore, the complex multiplicity of~$0$ is~$2^{[\frac{3}{2}]}=2$, and the complex multiplicity of each nonzero eigenvalue provided by~$\zeta^*$ is~$1$.
But note that different~$\zeta^*$'s may give rise to the same eigenvalue as long as they have the same length.
Note that with respect to this trivial spin structure,~$\Spect(\pd)$ is symmetric about~$0$.

Concerning the geometric Dirac operator~$\D$, which is essentially a double of~$\pd$, it is not surprising that \emph{$\D$ has the same eigenvalues as~$\pd$ but with the multiplicities doubled}.
Indeed, by writing
\begin{align}\label{eq:up-low}
    \psi=\begin{pmatrix}
        \psi^1 \\ \psi^2 \\ \psi^3 \\ \psi^4
    \end{pmatrix}
    =\begin{pmatrix}
        \psi_{\up} \\ \psi_{\low}
    \end{pmatrix}, \qquad
    \mbox{ with }
    \psi_{\up}=\begin{pmatrix}    \psi^1 \\ \psi^2 \end{pmatrix}
    \quad \mbox{ and } \quad
    \psi_{\low}=\begin{pmatrix}\psi^3 \\ \psi^4   \end{pmatrix},
\end{align}
we see that~$\D\psi=\lambda\psi$ is equivalent to
\begin{align}
    \begin{cases}
        \pd\psi_{\up}=&\lambda\psi_{\low}, \\
        \pd\psi_{\low}=&\lambda\psi_{\up}.
    \end{cases}
\end{align}
For each pair of eigenvalues~$\lambda_k$ and~$\lambda_{-k}=-\lambda_k$ in~$\Spect(\pd)$, let~$\varphi_k$ and~$\varphi_{-k}$ be the corresponding eigenspinors.
Then
\begin{align}
    \begin{pmatrix} \varphi_k \\ \varphi_k \end{pmatrix}, \begin{pmatrix} \varphi_{-k} \\ -\varphi_{-k}\end{pmatrix} \in \Eigen(\D; \lambda_k).
\end{align}
Hence~$\Spect(\pd)\subset\Spect(\D)$ and the multiplicities are at least doubled.
Conversely, if~$\D\psi=\lambda\psi$, then
\begin{align}
    \pd(\psi_{\up}+\psi_{\low}) = \lambda(\psi_{\up}+\psi_{\low}), & &
    \pd(\psi_{\up}-\psi_{\low}) = -\lambda(\psi_{\up}-\psi_{\low}),
\end{align}
and~$\psi_{\up}\pm \psi_{\low}$ cannot simultaneously vanish.
Hence~$\Spect(\D)\subset\Spect(\pd)$.

The counting on the multiplicity follows by a density argument.
Alternatively, we can count as follows. 
Note first that
\begin{align}
    \D=\begin{pmatrix} 0&\pd \\ \pd & 0 \end{pmatrix}.
\end{align}
After the unitary change of variables
\begin{align}
    \begin{pmatrix} \psi_{\up} \\ \psi_{\low}\end{pmatrix}
    \mapsto 
    \frac{1}{\sqrt2}
    \begin{pmatrix}
        \psi_{\up}+\psi_{\low}\\
        \psi_{\up}-\psi_{\low}
    \end{pmatrix},
\end{align}
the operator~$\D$ becomes~$\pd\oplus(-\pd)$. Therefore,
\begin{align}
    \Eigen(\D;\mu)
    \cong \Eigen(\pd;\mu)\oplus\Eigen(\pd;-\mu).
\end{align}
Since~$\Spect(\pd)$ is symmetric with equal multiplicities,
\begin{align}
    \Mtp(\D;\mu)=2\Mtp(\pd;\mu),\qquad \mu\neq0,
\end{align}
while~$\Mtp(\D;0)=2\Mtp(\pd;0)=4$.

It is crucial for us to understand the spectrum of~$\pD$.

\begin{lemma}\label{lemma:spectrum D}
    The spectrum of~$\pD$ is given by
    \begin{align}
        \Spect(\pD)
        =\braces{\pm\sqrt{\mu^2+ m^2} \; \mid \; \mu\in\Spect(\D)}
    \end{align}
    with multiplicities
    \begin{align}
        \Mtp&(\pD; \sqrt{\mu^2+m^2})= \Mtp(\pD; -\sqrt{\mu^2+m^2}) = \Mtp(\D; \mu),\quad \mbox{ for } \mu\neq 0, \\
        \Mtp&(\pD;m)=\Mtp(\pD;-m)=\frac{1}{2}\Mtp(\D;0)=2.
    \end{align}
\end{lemma}

\begin{proof}
If~$\lambda\in\Spect(\pD)$ has a nonzero eigenvector~$\psi\neq 0$, i.e.~$\pD\psi=\lambda\psi$, then
\begin{align}
    \lambda^2\psi = \pD^2 \psi = \D^2\psi+m^2\psi
\end{align}
so that~$\D^2\psi= (\lambda^2-m^2)\psi$.
We have seen that~$\D$ has eigenvalues~$\mu_j$'s, hence~$\D^2$ has eigenvalues~$\mu_j^2$'s.
Thus~$\lambda^2= \mu^2+m^2$ for some~$\mu\in\Spect(\D)$.
We observe that both~$\sqrt{\mu^2+m^2}$ and~$-\sqrt{\mu^2+m^2}$ appear as eigenvalues in the spectrum of~$\pD$.
Indeed, for~$\psi\in\Eigen(\D;\mu)$, consider the spinor~$\psi+t\gamma^0\psi$ with~$t\in\R$ to be determined. 
Then
\begin{align}
    \pD (\psi+ t\gamma^0\psi)
    =& \D\psi + m\gamma^0\psi + t\D\gamma^0\psi + tm \gamma^0\gamma^0\psi \\
    =&(\mu+tm)\psi+ (m-t\mu)\gamma^0\psi \\
    =&(\mu+tm)\parenthesis{\psi+ \frac{m-t\mu}{\mu+tm}\gamma^0\psi }.
\end{align}
It remains to solve the equation~$t=\frac{m-t\mu}{\mu+tm}$, whose roots are
\begin{align}
    t= \frac{-\mu\pm \sqrt{\mu^2+m^2}}{m}.
\end{align}
In particular,~$\mu+tm = \pm\sqrt{\mu^2+m^2}$, which would be an eigenvalue of~$\pD$ as long as~$\psi+t\gamma^0\psi\neq 0$.
If~$\mu\neq 0$, then~$t\neq \pm 1$, hence~$\psi+ t\gamma^0\psi\neq 0$.
This implies that any~$\mu\neq 0$ corresponds to two eigenvalues~$\braces{\pm\sqrt{\mu^2+m^2}}\subset\Spect(\pD)$, with
\begin{align}
    \pD \parenthesis{\psi + \frac{-\mu+\sqrt{\mu^2+ m^2}}{m}\gamma^0\psi}
    =& \sqrt{\mu^2+m^2} \parenthesis{\psi + \frac{-\mu+\sqrt{\mu^2+ m^2}}{m}\gamma^0\psi}, \\
    \pD \parenthesis{\psi + \frac{-\mu-\sqrt{\mu^2+ m^2}}{m}\gamma^0\psi}
    =& -\sqrt{\mu^2+m^2} \parenthesis{\psi + \frac{-\mu-\sqrt{\mu^2+ m^2}}{m}\gamma^0\psi}.
\end{align}
Let~$d=\Mtp(\D;\mu)$ for~$\mu>0$. The above construction gives
$d$ linearly independent eigenspinors for each of the eigenvalues
$\pm\sqrt{\mu^2+m^2}$.
On the other hand,
\begin{align}
    \Eigen(\D^2;\mu^2)
=\Eigen(\D;\mu)\oplus\Eigen(\D;-\mu)
\end{align}
has dimension~$2d$ and is invariant under~$\pD$.
Hence both eigenspaces of~$\pD$ have dimension exactly~$d$.

The zero eigenvalue of~$\D$ has eigenspinors given by the constant vectors
\begin{align}
    \psi_{01}=\begin{pmatrix}
        1\\ 0 \\ 0\\ 0
    \end{pmatrix}, \;
    \psi_{02}=\begin{pmatrix}
        0\\1\\0\\0
    \end{pmatrix}, \;
    \psi_{03}=\begin{pmatrix}
        0\\0\\1\\0
    \end{pmatrix}, \;
    \psi_{04}=\begin{pmatrix}
        0\\0\\0\\1
    \end{pmatrix}.
\end{align}
Then~$\psi_{0j}+\gamma^0\psi_{0j}$,~$j=1,2$, are eigenspinors of~$m$ for~$\pD$, while~$\psi_{0j}-\gamma^0\psi_{0j}$,~$j=3,4$, are eigenspinors of~$-m$ for~$\pD$.
That is, the 4-dimensional space~$\Eigen(\D;0)$ splits into~$\Eigen(\pD;m)\oplus\Eigen(\pD;-m)$, with each summand being 2-dimensional.
\end{proof}

\subsection{The working space}

With the above spectral information of~$\pD$, we can work on the fractional Sobolev space~$H^{\frac{1}{2}}(\mathbb{T}^3,\C^4)$, whose definition is recalled briefly in the following.

We have seen that~$\pD$ is a self-adjoint elliptic operator and the spectrum consists of nonzero eigenvalues.
Let~$\tau \colon \mathcal{B}(\R)\to \operatorname{Proj}(L^2(\mathbb{T}^3,\C^4))$ be the spectral measure of~$\pD$, where~$\mathcal{B}(\R)$ is the Borel~$\sigma$-algebra of~$\R$.
Then the operator~$\pD$ admits the spectral resolution
\begin{align}
    \pD =\int_\R \lambda \dd\tau(\lambda).
\end{align}

The absolute value of \(\pD\), denoted by \(|\pD|\), is a non-negative self-adjoint operator defined by
\[
|\pD| = \int_{\mathbb{R}} |\lambda| \, \dd\tau(\lambda),
\]
with domain \(\Dom(|\pD|) := \left\{ \psi \in L^2(\mathbb{T}^3, \mathbb{C}^4) \mid \int_{\mathbb{R}} |\lambda|^2 \, \dd\Abracket{\tau(\lambda)\psi,\psi}_{L^2} < \infty \right\}\).
The operator~$|\pD|^{\frac{1}{2}}$ is similarly defined by
\[
|\pD|^{1/2} = \int_{\mathbb{R}} |\lambda|^{1/2} \, \dd\tau(\lambda),
\]
with domain
\begin{align}
   \Dom(|\pD|^{1/2}) = \left\{ \psi \in L^2(\mathbb{T}^3, \mathbb{C}^4) \mid \int_{\mathbb{R}} |\lambda| \, \dd\Abracket{\tau(\lambda)\psi,\psi}_{L^2}  < \infty \right\}.
\end{align}

We will work with the space
\begin{align}
    H^{\frac{1}{2}}(\mathbb{T}^3, \mathbb{C}^4)
    \coloneqq \braces{ \psi\in L^2(\mathbb{T}^3;\C^4)\;\mid \; \|\psi\|_{H^{\frac{1}{2}}} \equiv \parenthesis{\|\psi\|_{L^2}^2 + \||\pD|^{\frac{1}{2}}\psi\|_{L^2}^2}^{\frac{1}{2}} <+\infty }.
\end{align}
Note that for each~$\psi\in  H^{\frac12}(\mathbb{T}^3, \mathbb{C}^4)$, we have
  \begin{align}
      \int_{\mathbb{T}^3} ||\pD|^{1/2}\psi|^2 \dv 
      = \langle |\pD|^{1/2}\psi, |\pD|^{1/2}\psi \rangle_{L^2(\mathbb{T}^3, \mathbb{C}^4)}
      = \langle |\pD|\psi, \psi \rangle_{H^{-\frac{1}{2}}\times H^{\frac{1}{2}}} 
      = \int_{\mathbb{R}} |\lambda| \,\dd\Abracket{\tau(\lambda)\psi,\psi}_{L^2} .
  \end{align}
Since the smallest eigenvalue of~$|\pD|$ is~$m>0$, we see that 
\begin{align}
    \|\psi\|_{L^2} \leq \frac{1}{\sqrt{m}}\| |\pD|^{\frac{1}{2}} \psi\|_{L^2}. 
\end{align}
Thus we have an equivalent norm on~$H^{\frac{1}{2}}(\mathbb{T}^3,\C^4)$ given by 
\begin{align}\label{eq:equiv norm}
    \tnorm{\psi}_{H^{1/2}(\T^3,\C^4)}:= \| |\pD|^{\frac{1}{2}} \psi\|_{L^2(\mathbb{T}^3,\C^4)}= \parenthesis{ \langle |\pD|\psi, \psi \rangle_{H^{-\frac{1}{2}}\times H^{\frac{1}{2}}} }^{\frac{1}{2}}.
\end{align}
We remark that for~$\psi\in H^{\frac{1}{2}}(\mathbb{T}^3,\C^4)$, the expression~$\int_{\T^3}\Abracket{\D\psi,\psi}\dv$ will always be understood in the duality pairing sense.

According to the spectral information of~$\pD$, we define the closed subspaces spanned by eigenspinors corresponding to positive and negative eigenvalues, respectively, as
\begin{align}
    H^{\frac{1}{2},+}\coloneqq 
    \overline{\bigoplus_{\lambda>0} \Eigen(\pD; \lambda)}^{H^{\frac{1}{2}}}, 
    \quad \mbox{resp.}
    \quad 
    H^{\frac{1}{2},-}\coloneqq 
    \overline{\bigoplus_{\lambda<0} \Eigen(\pD; \lambda)}^{H^{\frac{1}{2}}}. 
\end{align}
Then we have the following orthogonal decomposition 
\begin{align}
    H^{\frac{1}{2}}(\mathbb{T}^3, \mathbb{C}^4)
    =H^{\frac{1}{2},+}\oplus H^{\frac{1}{2},-}.
\end{align}
Furthermore, we let
\begin{align}
    P^{\pm} \colon H^{\frac{1}{2}}(\mathbb{T}^3,\C^4) \to H^{\frac{1}{2},\pm}
\end{align}
denote the orthogonal projections onto the subspaces~$H^{\frac{1}{2},\pm}$. 
For any~$\psi\in H^{\frac{1}{2}}$, we can decompose it as 
\begin{align}
    \psi = P^+\psi + P^-\psi \equiv \psi^+ + \psi^- \in H^{\frac{1}{2},+}\oplus H^{\frac{1}{2},-}. 
\end{align}
Note that this is also a~$L^2$ orthogonal projection: 
\begin{align}
    \|\psi\|_{L^2}^2 = \|\psi^+\|_{L^2}^2 + \|\psi^-\|_{L^2}^2.
\end{align}

With the above projectors, we have 
\begin{align}
    \pD = \pD(P^+ + P^-) = \pD P^+ + \pD P^-, & &\mbox{ and }& &
    |\pD|= \pD P^+ - \pD P^-.
\end{align}


\section{The variational structure}\label{sect:variational str}

The equation~\eqref{eq:NDE-spatial} is the Euler-Lagrange equation of the functional
\begin{align}
    J\colon H^{\frac{1}{2}}(\T^3,\C^4)\to \R
\end{align}
given by
\begin{align}
    J(\psi)
    =\int_{\mathbb{T}^3}  \frac{1}{2}\Abracket{\psi,\pD\psi} -\frac{a}{2}|\psi|^2 - F(x,\psi)\dv.
\end{align}
Remark that, throughout the variational argument, we regard $H^{\frac12}(\mathbb{T}^3,\C^4)$ as a real Hilbert space, and all derivatives and pairings are understood in the real sense.
Recall that the nonlinearity~$F$ satisfies the hypotheses (F1-5).
In particular,~$F(x,\psi)$ fails to be coercive in~$\psi$ in general.
Actually, in the model example,~$F(\psi)=\frac{1}{2}G(\bar{\psi}\psi)$, 
using the notation from~\eqref{eq:up-low}, we have
\begin{align}
    \bar{\psi}\psi = \Abracket{\gamma^0\psi,\psi}= |\psi_{\up}|^2 - |\psi_{\low}|^2, & &
    |\psi|^2 = |\psi_{\up}|^2 + |\psi_{\low}|^2 = \sum_{k=1}^4 |\psi^k|^2.
\end{align}
Note that~$\bar{\psi}\psi$ may be small, while~$|\psi|^2$ is large.
This kind of nonlinearity frequently arises in various particle models in quantum field theory, which is more complicated than the nonlinearities of the form~$G(|\psi|^2)$.
The approach we take here is motivated by~\cite{Bartsch1999nonlinear, EstebanSere1995stationary, HoferWysocki1990first,Sere1995homoclinic}.
Aside from the strongly indefinite nature of Dirac operators, there are new challenges in the periodic setting, such as the construction of the min-max structure, and the absence of a useful radial ansatz.
Methodologically, the powerful Fourier transform on~$\R^3$, which provides rich variational information in~\cite{EstebanSere1995stationary}, is no longer available.
Although Fourier series are still available on~$\mathbb T^3$, constructing a suitable linking structure requires the detailed spectral analysis of~$\pD$. 
We aim to fix these issues and look for new variational periodic solutions.

\

Ding and Liu~\cite{DingLiu2014periodic} also worked on periodic solutions for nonlinear Dirac equations.
The nonlinearities in their consideration admit a lower bound in the form of Ambrosetti-Rabinowitz type, namely of the form~$G(|\psi|^2)$ for a suitable function~$G$.
By detecting the variational structure near~$0$ and near infinity, they succeeded in finding min-max periodic solutions.
For the nonlinearity considered here, which fails to be coercive globally, we have to employ a perturbation method to obtain the desired variational periodic solution.

We observe that, with the help of the spectral analysis in Lemma \ref{lemma:spectrum D}, the functional~$J$ admits a local linking structure as in~\cite{EstebanSere1995stationary}.
However, due to the loss of the global coercivity of the nonlinearity~$F$ in~$\psi$, we cannot verify the Palais--Smale condition for~$J$.
Fortunately we can still employ the perturbation method in~\cite{EstebanSere1995stationary} and first obtain perturbative solutions, then establish uniform estimates of such solutions, finally take the limit of a subsequence which serves as a nontrivial periodic solution.


\section{Perturbations}\label{sect:perturbation}

For each~$\eps\in [0,1]$, consider the functional
\begin{align}
    J_\eps(\psi)
    =& J(\psi)-\eps\int_{\mathbb{T}^3} |\psi|^{\alpha_2}\dv \\
    =& \int_{\mathbb{T}^3} \frac{1}{2}\Abracket{\psi,\pD\psi}-\frac{a}{2}\Abracket{\psi,\psi}- F(x,\psi)-\eps|\psi|^{\alpha_2}\dv.
\end{align}
We will abbreviate~$F(x,\psi)+\eps|\psi|^{\alpha_2}$ as~$F_{\eps}(x,\psi)$.
Note that the functions~$F_\eps$ satisfy similar hypotheses as~$F$ with possibly different constants~$A_j$.
With an abuse of notation let us assume that 
\emph{ (F1-4) are satisfied by~$F_\eps$ for any~$\eps\in [0,1]$}.
Remark that we do not require the perturbed~$F_\eps$ to satisfy (F5); condition (F5) is used only for the unperturbed nonlinearity~$F$ in the uniform-estimate argument.
The Euler-Lagrange equation of~$J_\eps$ is 
\begin{align}\label{eq:perturbative Dirac eq}
    \pD\psi -a\psi -\p_\psi F(x,\psi)=\eps \alpha_2 |\psi|^{\alpha_2-2}\psi. 
\end{align}

\subsection{Verification of the Palais--Smale condition for the perturbed functionals}
Fix~$\eps \in (0,1]$, and let~$(\psi_n)_{n\geq 1}$ be a Palais-Smale sequence for~$J_\eps$ at a level~$c$, namely 
\begin{align}\label{eq:PS-level}
    J_\eps(\psi_n)=\int_{\mathbb{T}^3}  \frac{1}{2}\langle \psi_n, \pD\psi_n\rangle - \frac{a}{2}|\psi_n|^2 - F_{\varepsilon}(x,\psi_n) \dv \to c, 
\end{align}
\begin{align}\label{eq:PS:differential}
    \dd J_\eps(\psi_n)= \pD\psi_n - a\psi_n - \p_\psi F_{\varepsilon}(x,\psi_n) \; \to \; 0, 
    \qquad \mbox{ in } H^{-\frac{1}{2}}
\end{align}
as~$n\to +\infty$. 
We need to show that there exists a subsequence converging in~$H^{\frac{1}{2}}$ to a limit~$\psi_\infty$ which is a weak solution of~\eqref{eq:perturbative Dirac eq}.

We first show that the sequence~$(\psi_n)$ is bounded in~$H^{\frac{1}{2}}$. 
To see this, we test~\eqref{eq:PS:differential} against~$\psi_n$ and take the difference with~\eqref{eq:PS-level}, then by (F3)
\begin{align}
    2J_\eps(\psi_n)-\dd J_\eps(\psi_n)[\psi_n]
    =&\int_{\mathbb{T}^3} \p_\psi F(x,\psi_n)[\psi_n]-2F(x,\psi_n)+\eps (\alpha_2-2)|\psi_n|^{\alpha_2}\dv \\
    \geq & \eps(\alpha_2-2)\int_{\mathbb{T}^3} |\psi_n|^{\alpha_2}\dv, 
\end{align}
that is, 
\begin{align}
    \int_{\mathbb{T}^3} |\psi_n|^{\alpha_2}\dv \leq 
    \frac{1}{\eps(\alpha_2-2)} \parenthesis{ 2c + o(1) + o(\|\psi_n\|_{H^{\frac{1}{2}}})}.
\end{align}
As a consequence, in the equation 
\begin{align}
    \pD \psi_n = a\psi_n + \p_\psi F(x,\psi_n) + \eps \alpha_2 |\psi_n|^{\alpha_2-2}\psi_n +\dd J_\eps(\psi_n),  
\end{align}
the right-hand side has the bound 
\begin{align}
    \| RHS\|_{H^{-\frac{1}{2}}}
    \leq & a \|\psi_n\|_{H^{-\frac{1}{2}}} +C(A_2,\alpha_2)(1+ \||\psi_n|^{\alpha_2-1}\|_{H^{-\frac{1}{2}}} )+ o(1) \\
    \leq & C(a, A_2, \eps, \alpha_2, \vol(\mathbb{T}^3))(1+ \|\psi_n\|_{L^{\alpha_2}}^{\alpha_2-1} )  + o(1) \\
    \leq&  C(a, A_2, \eps, \alpha_2, \vol(\mathbb{T}^3)) \parenthesis{1+ \frac{1}{\eps(\alpha_2-2)} \parenthesis{ 2c + o(1) + o(\|\psi_n\|_{H^{\frac{1}{2}}})}}^{\frac{\alpha_2-1}{\alpha_2}}.
\end{align}

Since~$\pD$ has no kernel, 
\begin{align}
    \|\psi_n\|_{H^{\frac{1}{2}}} \leq \|\pD \psi_n\|_{H^{-\frac{1}{2}}}.
\end{align}
It follows that
\begin{align}
    \|\psi_n\|_{H^{\frac{1}{2}}} \leq C(c,a, A_2, \eps,\alpha_2,\vol(\mathbb{T}^3)).
\end{align}
Here we note that~$\alpha_2>2$ and~$\eps>0$ are essential for this estimate, which blows up as~$\eps\to 0^+$.
That is also the reason for introducing this perturbation, otherwise we cannot be sure of the boundedness of Palais-Smale sequences for~$J$. 

\ 

Since~$H^{\frac{1}{2}}(\mathbb{T}^3,\C^4)$ is a Hilbert space, we can extract a \emph{weakly} convergent subsequence which will still be denoted by~$\psi_n$, with limit~$\psi_\infty\in H^{\frac{1}{2}}$.
It remains to show that~$\psi_\infty$ is a critical point of~$J_\eps$ and the convergence is actually strong.

By the compactness of the Sobolev embedding $H^{\frac{1}{2}}(\mathbb{T}^3)\hookrightarrow L^q(\mathbb{T}^3)$ for $q<3$ and~$\mathbb{T}^3$ closed, we have
\[
\psi_{n}\to\psi_\infty\quad\text{strongly in }L^q,\quad\forall q<3.
\]
In particular,~$\psi_n\to\psi_\infty$ in~$L^{\alpha_2}$. 
Using~$\big||z|^{\alpha_2-2}z-|w|^{\alpha_2-2}w\big| \leq C(|z|+|w|)^{\alpha_2-2}|z-w| $, we obtain
\[
|\psi_n|^{\alpha_2-2}\psi_n
\to
|\psi_\infty|^{\alpha_2-2}\psi_\infty
\quad\text{strongly in }L^{\frac{\alpha_2}{\alpha_2 -1}}.
\]
By (F2) and continuity of $\partial_\psi F$, $|\partial_\psi F(x,\psi)| \leq C(1+|\psi|^{\alpha_2-1})$. 
Since \(\psi_n\to\psi_\infty\) strongly in \(L^{\alpha_2}\), the continuity of the associated Nemytskii operator gives
\[
\partial_\psi F(\cdot,\psi_n)
\to
\partial_\psi F(\cdot,\psi_\infty)
\quad\text{strongly in }L^{\frac{\alpha_2}{\alpha_2 -1}}.
\]

For any test function $\phi\in C^\infty(\mathbb{T}^3,\C^4)$,
\[
\dd J_\eps(\psi_{n})[\phi]=\int_{\mathbb{T}^3}\Abracket{\phi,\;\pD\psi_{n}-a\psi_{n}-\p_\psi F(x,\psi_{n})-\eps\alpha_2|\psi_{n}|^{\alpha_2-2}\psi_{n}}\dv= o(\|\phi\|)\to 0.
\]
As $n\to\infty$, we have
\begin{itemize}
    \item $\displaystyle\int_{\mathbb{T}^3}\Abracket{\phi,\pD\psi_{n}}\dv\to\int_{\mathbb{T}^3}\Abracket{\phi,\pD\psi_\infty}\dv$ by weak $H^{\frac{1}{2}}$-convergence;
    \item $\displaystyle\int_{\mathbb{T}^3}\Abracket{\phi,a\psi_{n}}\dv\to\int_{\mathbb{T}^3}\Abracket{\phi,a\psi_\infty}\dv$ by weak (actually strong) $L^2$-convergence;
    \item $\displaystyle\int_{\mathbb{T}^3}\Abracket{\phi,\p_\psi F(x,\psi_{n})}\dv\to\int_{\mathbb{T}^3}\Abracket{\phi,\p_\psi F(x,\psi_\infty)}\dv$ by strong $L^{\frac{\alpha_2}{\alpha_2 -1}}$-convergence of the nonlinear potential part~$\p_\psi F(\cdot,\psi_n)$; 
    \item $\displaystyle\eps\alpha_2\int_{\mathbb{T}^3}\Abracket{\phi,|\psi_n|^{\alpha_2-2}\psi_n}\dv\to\eps\alpha_2\int_{\mathbb{T}^3}\Abracket{\phi,|\psi_{\infty}|^{\alpha_2-2}\psi_{\infty}}\dv$ by the~$L^{\frac{\alpha_2}{\alpha_2 -1}}$ convergence of perturbation term~$|\psi_n|^{\alpha_2-2}\psi_n$.
\end{itemize}
Hence, we can pass to the limit and obtain
\[
\int_{\mathbb{T}^3}\Abracket{\phi,\;\pD\psi_\infty-a\psi_\infty-\p_\psi F(x,\psi_\infty)-\eps\alpha_2|\psi_\infty|^{\alpha_2-2}\psi_\infty}\dv=0,\qquad\forall\phi\in C^\infty(\mathbb{T}^3,\C^4). 
\]
That is, $\dd J_\varepsilon(\psi_\infty)=0$.

Finally, since~$\psi_n- \psi_\infty$ converges to~$0$ weakly in~$H^{\frac{1}{2}}$ and strongly in~$L^q$ for any~$1<q<3$, we have  
\begin{align}
    \pD(\psi_n-\psi_\infty) 
    =a(\psi_n-\psi_\infty)+\p_\psi F(x,\psi_n)-\p_\psi F(x,\psi_\infty) + \eps \alpha_2 \parenthesis{|\psi_n|^{\alpha_2-2} \psi_n - |\psi_\infty|^{\alpha_2-2}\psi_\infty }  + \dd J_\eps(\psi_n) 
\end{align}
and the right hand side above converges to~$0$ in~$H^{-\frac{1}{2}}$, hence 
\begin{align}
    \|\psi_n-\psi_\infty\|_{H^{\frac{1}{2}}} \leq C\|\pD(\psi_n-\psi_\infty)\|_{H^{-\frac{1}{2}}} \to 0.
\end{align}

Therefore, the Palais-Smale condition holds for~$J_\eps$ for each~$\eps\in (0,1]$.

\subsection{Linking structures for perturbed functionals}
We now show that~$J_\eps$ admits a local linking structure in a uniform manner.
Let
\begin{align}
    \mathscr{P}= \Span_{\C}\braces{ \begin{pmatrix}1 \\ 0 \\ 0 \\ 0 \end{pmatrix}, \; \begin{pmatrix} 0\\ 1\\ 0\\ 0 \end{pmatrix} }
\end{align}
which is a vector space of complex dimension 2.
It consists of parallel spinors on~$\mathbb{T}^3$ whose lower components are zero.
The key feature is that
\begin{align}
    \gamma^0\psi=\psi, \qquad \forall \psi\in\mathscr{P}.
\end{align}
Recall the equivalent norm given in~\eqref{eq:equiv norm}.

\begin{lemma}\label{lemma:local control}
    For any~$\mu \in  (0,m)$ and any~$N\in\braces{1,2,3,4}$, there exists a real~$N$-dimensional space~$\mathscr{E}_N\subset H^{\frac{1}{2},+}$ and a positive number~$R>0$ such that for any~$\psi\in H^{\frac{1}{2}}$,
       \begin{multline}
         \parenthesis{ \tnorm{P^-\psi}_{H^{1/2}}\leq R, \; \mbox{ and }\;  P^+\psi\in\mathscr{E}_N, \, \tnorm{P^+\psi}_{H^{1/2}}=R }\\
        \quad \Longrightarrow  \quad 
        \quad \int_{\mathbb{T}^3} \frac{1}{2}\Abracket{\psi,\pD\psi}-F_\eps(x,\psi)\dv \leq \int_{\mathbb{T}^3} \frac{\mu}{2}|\psi|^2 \dv.
    \end{multline}
\end{lemma}
Note that~$R$ is chosen independently of~$N$ and~$\eps$.

\begin{proof}[Proof of Lemma~\ref{lemma:local control}]

Let~$\mathscr{E}_N$ be a \emph{real}~$N$-dimensional subspace of~$\mathscr{P}$, which consists of constant spinors and has zero lower components, i.e.~$\psi_{\low}=0$.
In particular, for any~$e\in\mathscr{E}_N$, we have
\begin{align}
    \pD e = \left(-i\gamma^0 \gamma^k \partial_k + m\gamma^0\right)e = \D e + m\gamma^0 e = me,
\end{align}
\begin{align}
    P^+ e = e, \quad P^-\gamma^0 e = P^- e=0,
\end{align}
and
\begin{align}
    \bar{e}e = \Abracket{\gamma^0 e, e} = \Abracket{e,e}=|e|^2=|e_{\up}|^2.
\end{align}

Thus, setting~$e=P^+\psi$ with~$\tnorm{e}_{H^{1/2}}=R$, we have
\begin{align}
    \int_{\mathbb{T}^3} \frac{1}{2}\Abracket{e,\pD e}\dv
    =\int_{\mathbb{T}^3} \frac{m}{2}|e|^2\dv
    =\frac{1}{2}\tnorm{e}_{H^{{1}/{2}}}^2
    = \frac{1}{2}R^2.
\end{align}
To obtain the desired result, we notice that~$F_\eps(\psi)\geq 0$. 
Thus if
\begin{align}
    \tnorm{P^-\psi}^2_{H^{{1}/{2}}}=\int_{\T^3}\Abracket{|\pD|P^-\psi,P^-\psi}\dv=-\int_{\mathbb{T}^3} \Abracket{P^-\psi, \pD P^-\psi}\dv
    \geq (1-\frac{\mu}{m})R^2,
\end{align}
then we immediately have
\begin{align}
    \int_{\mathbb{T}^3}\frac{1}{2}\Abracket{\psi,\pD\psi}-F_\eps(\psi)\dv
    \leq &\frac{1}{2}\int_{\mathbb{T}^3} \Abracket{P^+\psi, \pD P^+\psi} + \Abracket{P^-\psi,\pD P^-\psi} \dv \\
    \leq &\frac{\mu}{2m}R^2
    = \frac{\mu}{2}\int_{\mathbb{T}^3} |e|^2\dv
    \leq \frac{\mu}{2}\int_{\mathbb{T}^3} |\psi|^2\dv.
\end{align}
Otherwise,~$\tnorm{P^- \psi}^2_{H^{{1}/{2}}} < (1-\frac{\mu}{m})R^2$, by Lemma~\ref{lemma:spectrum D} we have
\begin{align}
    \|P^-\psi\|_{L^2}^2 \leq \frac{1}{m} \tnorm{P^-\psi}_{H^{{1}/{2}}}^2 \leq \frac{1}{m}\parenthesis{1-\frac{\mu}{m}}R^2.
\end{align}
Then, using (F4), we see that
\begin{align}
    -\int_{\mathbb{T}^3} F_\eps(x,\psi)\dv
    \leq& -\int_{\mathbb{T}^3} A_3 |\bar{\psi}\psi|^\nu -A_4 \dv
    =-A_3\int_{\mathbb{T}^3} |\bar{\psi}\psi|^\nu\dv + A_4 \cdot\vol(\mathbb{T}^3).
\end{align}
The term involving~$\bar{\psi}\psi$ can be estimated by
\begin{align}
    \parenthesis{\int_{\mathbb{T}^3} |\bar{\psi}\psi|^\nu\dv }^{\frac{1}{\nu}}\parenthesis{\int_{\mathbb{T}^3} 1 \dv}^{1-\frac{1}{\nu}}
    \geq & \int_{\T^3} \bar{\psi}\psi\dv  \\
    =& \int_{\mathbb{T}^3} \Abracket{\gamma^0\parenthesis{P^+\psi+ P^-\psi}, P^+\psi + P^-\psi}\dv \\
    =& \int_{\mathbb{T}^3} |e|^2 + \Abracket{\gamma^0 P^- \psi, P^-\psi}\dv \\
    \geq & \| e\|_{L^2}^2 - \|P^-\psi\|_{L^2}^2 \\
    \geq & \frac{1}{m} R^2 -\frac{1}{m}(1-\frac{\mu}{m})R^2\\
    \geq & \frac{\mu}{2m^2}R^2,
\end{align}
that is,
\begin{align}
    \int_{\mathbb{T}^3} |\bar{\psi}\psi|^\nu \dv
    \geq  \frac{\mu^\nu}{\vol(\mathbb{T}^3)^{\nu -1}} \parenthesis{\frac{R^2}{2m^2}}^{\nu}.
\end{align}
These are essentially the estimates in~\cite[Lemma 2.1]{EstebanSere1995stationary}, but note that on the torus it becomes more clear due to the presence of constant eigenspinors.
Now we can estimate
\begin{align}
    \int_{\mathbb{T}^3} \frac{1}{2}\Abracket{\psi,\pD \psi}- F_\eps(x,\psi) \dv
    \leq & \frac{1}{2}\tnorm{P^+\psi}^2_{H^{{1}/{2}}} - \frac{1}{2}\tnorm{P^-\psi}^2_{H^{{1}/{2}}} - \int_{\mathbb{T}^3} F_\eps(x,\psi)\dv \\
    \leq & \frac{1}{2}R^2 - \frac{A_3\mu^\nu}{\vol(\mathbb{T}^3)^{\nu-1} (2m^2)^\nu} R^{2\nu} + A_4 \vol(\mathbb{T}^3).
\end{align}
Since~$\nu>1$ and~$A_3>0$, we can choose~$R=R(m,\vol(\mathbb{T}^3),A_3, A_4,\mu,\nu)>0$ such that the RHS above is negative.
Note that this choice of~$\mathscr{E}_N$ and~$R$ is independent of~$\eps\in [0,1]$.
\end{proof}
\begin{rmk}
    Later we will apply Lemma~\ref{lemma:local control} with~$\mu=a$.
    For such a choice it suffices to have
    \begin{align}
        R\geq \max \braces{ \sqrt{2A_4\vol(\mathbb{T}^3)}, \parenthesis{\frac{2^\nu}{A_3}}^{\frac{1}{2(\nu-1)}}\sqrt{\vol(\mathbb{T}^3)} \parenthesis{\frac{m}{\sqrt{a}}}^{\frac{\nu}{\nu-1}} }.
    \end{align}
    We see that~$R$ can be small if~$\vol(\mathbb{T}^3)$ is small. 
    Moreover, the closer~$a$ is to~$m$, the smaller this upper bound can be. 
\end{rmk}

Next we show that~$J_\eps$ admits a linking structure. 
We consider the case~$N=1$. 
In this case, let~$\mathscr{E}_1=\Span(e)$ with~$\tnorm{e}_{H^{{1}/{2}}}=\sqrt{m}\|e\|_{L^2}=1$, and let 
\begin{align}
    \mathscr{C}_1(R)
    \coloneqq \braces{\psi= \psi^- + \lambda e\in H^{\frac{1}{2}}(\mathbb{T}^3,\C^4) \mid  \psi^-\in H^{\frac{1}{2},-}, \; \tnorm{\psi^-}_{H^{1/2}}\leq R, \; \lambda \in [0,R]}. 
\end{align}
For~$r\in (0,R)$, consider the sphere in~$H^{\frac{1}{2},+}$ with radius~$r$:
\begin{align}
    \mathscr{S}^+(r)\coloneqq \braces{\psi\in H^{\frac{1}{2},+} \; \mid \; \tnorm{\psi}_{H^{{1}/{2}}}=r }.
\end{align}
Note that both of these sets are infinite-dimensional, hence we cannot claim that they link in the usual sense of~\cite{Ambrosetti2007Nonlinear,Benci1979Critical}. 
Nevertheless, we will see that they separate the functional~$J_\eps$ in the sense of~\cite{Schechter2021linking}. 

\begin{lemma}\label{lemma:negativity on boundary-1}
    For any~$\eps\in [0,1]$, we have~$J_\eps|_{\p\mathscr{C}_1(R)}\leq 0$.
\end{lemma}
\begin{lemma}\label{lemma:positivity in interior-1}
    There exists~$r\in (0,R)$ and~$C_*>0$ such that for any~$\eps\in [0,1]$, we have $J_\eps|_{\mathscr{S}^+(r)}\geq C_*$. 
\end{lemma}

\begin{proof}[Proof of Lemma~\ref{lemma:negativity on boundary-1}]
    Let~$\psi\in\p\mathscr{C}_1(R)$, namely, either~$\tnorm{\psi^-}_{H^{1/2}}=R$ or~$\lambda\in \braces{0,R}$.
    \begin{itemize}
        \item If~$\tnorm{\psi^-}_{H^{1/2}}=R$, then for any~$\lambda\in[0,R]$, we have
            \begin{align}
                J_\eps(\psi^-+\lambda e)\leq \frac{1}{2}(\lambda^2 - R^2)\leq 0.
            \end{align}
        \item If~$\lambda=0$, then~$J_\eps(\psi^-)\leq 0$. 
        \item If~$\lambda=R$, then~$\tnorm{\psi^-}_{H^{1/2}}\leq R$ and~$\tnorm{\psi^+}_{H^{1/2}}= R\tnorm{e}_{H^{1/2}}=R$; by Lemma~\ref{lemma:local control} (with~$\mu=a$) we again have~$J_\eps(\psi)\leq 0$. 
    \end{itemize}
\end{proof}

\begin{proof}[Proof of Lemma~\ref{lemma:positivity in interior-1}]
    By (F1), for~$\psi^+\in H^{\frac{1}{2},+}$ we have 
    \begin{align}
        J_\eps(\psi^+)
        \geq & \frac{1}{2}\int_{\mathbb{T}^3} \Abracket{\psi^+, \pD\psi^+} - a|\psi^+|^2 \dv 
        - A_1\int_{\mathbb{T}^3} |\psi^+|^{\alpha_1} + |\psi^+|^{\alpha_2}\dv \\
        \geq & \frac{1}{2}\parenthesis{1-\frac{a}{m}}\tnorm{\psi^+}^2_{H^{1/2}} - C \parenthesis{ \tnorm{\psi^+}^{\alpha_1}_{H^{1/2}}+\tnorm{\psi^+}^{\alpha_2}_{H^{1/2}}}.
    \end{align}
    Since~$a<m$ and~$2<\alpha_1\leq\alpha_2<3$ by assumption, we can choose~$r$ so small that the quadratic part dominates the superquadratic part, hence the conclusion follows.
\end{proof}

We note that~$r$ and~$C_*$ can also be chosen to be independent of~$\eps\in [0,1]$.
Furthermore, 
\begin{align}\label{eq:intersection}
    \mathscr{C}_1(R)\cap \mathscr{S}^+(r)= \braces{ r e} \neq\emptyset. 
\end{align}
It follows that 
\begin{align}
    \sup_{\p \mathscr{C}_1(R)} J_\eps =  0<C_*\leq  \inf_{\mathscr{S}^+(r)} J_\eps \leq\sup_{\mathscr{C}_1(R)} J_\eps \equiv C^*<+\infty
\end{align}
for each~$\eps\in [0,1]$. 
That is,~$\p\mathscr{C}_1(R)$ and~$\mathscr{S}^+(r)$ separate the functionals~$J_\eps$.
We will later make an explicit estimate of~$C^*$.

\subsection{Deformations and minimax level}

Next we need a suitable class of deformations. 
Note that we cannot directly use the gradient flow as in~\cite{EstebanSere1995stationary}.
Indeed, on the space~$H^{\frac{1}{2}}(\mathbb{T}^3,\C^4)$ we have the norm~$\tnorm{\cdot}_{H^{1/2}}$ and the inner product 
\begin{align}
    \ll\psi,\varphi\gg \coloneqq \int_{\mathbb{T}^3} \Abracket{|\pD|^{\frac{1}{2}}\psi,|\pD|^{\frac{1}{2}}\varphi}\dv
    =\int_{\mathbb{T}^3} \Abracket{|\pD|\psi,\varphi}\dv.
\end{align}
With respect to this inner product structure the gradient of~$J_\eps$ has the form 
\begin{align}
    \nabla J_\eps (\psi)
    = &|\pD|^{-1}\parenthesis{\pD\psi - a\psi-\p_\psi F(\cdot,\psi) - \eps \alpha_2 |\psi|^{\alpha_2 -2 }\psi} \\
    =& (P^+ - P^-)\psi - a |\pD|^{-1}\psi - |\pD|^{-1} \p_\psi F(\cdot,\psi) - \alpha_2 \eps |\pD|^{-1}\parenthesis{ |\psi|^{\alpha_2 -2}\psi}.
\end{align}
We have seen that the term~$|\pD|^{-1} \p_\psi F(\cdot,\psi)$ is only H\"older continuous but in general not Lipschitz. 
Thus the gradient flow may not exist globally.

Note that the maps~$K_\eps\colon H^{\frac{1}{2}} \to H^{\frac{1}{2}}$ defined by
\begin{align}
    K_\eps(\psi) \equiv a|\pD|^{-1}\psi + |\pD|^{-1} \p_\psi F(\cdot,\psi) + \alpha_2 \eps |\pD|^{-1}\parenthesis{ |\psi|^{\alpha_2 -2}\psi}
\end{align}
are compact and continuous.
Hence we can use certain pseudo-gradient vectors and pseudo-gradient flows as a replacement. 
In the following part of this section,~$\eps\in (0,1]$ will be fixed. 

\begin{lemma}\label{lemma:global compact approx}
    For any~$\delta>0$, there exists a locally Lipschitz map
    \begin{align}
        K_{\eps,\delta}\colon H^{\frac{1}{2}} \to H^{\frac{1}{2}}
    \end{align}
    such that 
    \begin{itemize}
        \item[(K1)] for each~$R>0$, $K_{\eps,\delta}(\mathscr{B}(R))$ is bounded and contained in a finite-dimensional subspace of~$H^{\frac{1}{2}}$ (here~$\mathscr{B}(R)$ is the ball of radius~$R$ in~$H^{\frac{1}{2}}$); 
        \item[(K2)] $K_{\eps,\delta}$ is~$\delta$-close to~$K_\eps$ in the sense that 
                    \begin{align}
                        \sup_{\psi\in H^{\frac{1}{2}}} \tnorm{ K_{\eps,\delta}(\psi)- K_{\eps}(\psi) }_{H^{1/2}} \leq \delta. 
                    \end{align}
    \end{itemize}
\end{lemma}

\begin{proof}
    This is a standard locally Lipschitz approximation of the compact operator~$K_{\eps}$, see e.g.~\cite{Wojciech1998Generalized}. 
    We give a proof for completeness. 

    For each~$n\geq 1$, the set~$\overline{K_\eps (\mathscr{B}(n+2))}$ is compact, hence admits a finite~$\frac{\delta}{4}$-net, say~$y_{n,1},y_{n,2},\dots, y_{n, N_n}$.
    As~$K_\eps$ is continuous, we can take a finite open cover of~$\mathscr{B}(n+2)$ such that on each open set in that cover,~$K_\eps$ differs by at most~$\frac{\delta}{4}$ from one of the points~$y_{n,j}$.
    Then a locally Lipschitz partition of unity subordinate to this finite cover gives a locally Lipschitz map 
    \begin{align}
        K_{\eps,n}\colon \mathscr{B}(n+2) \to \Span\braces{y_{n,1},\cdots, y_{n,N_n}}
    \end{align}
    with~$\tnorm{K_\eps(\psi)-K_{\eps,n}(\psi)}_{H^{1/2}}<\delta/2$ for any~$\psi\in\mathscr{B}(n+2)$.

    Next we patch these~$K_{\eps,n}$ together to get a global approximation. 
    Choose a locally finite Lipschitz partition of unity~$\braces{\chi_n}_{n\geq 1}$ on~$[0,+\infty)$ such that~$\sum_{n\geq 1} \chi_n\equiv 1$. 
    We may assume~$\spt(\chi_n)\subset (n-1,n+1)$ for~$n\geq 2$ and~$\spt(\chi_1)\subset [0,2)$ and~$\chi_1\equiv 1$ on~$[0,1]$. 
    Then we can simply take 
    \begin{align}
        K_{\eps,\delta} (\psi) = \sum_{n=1}^{+\infty} \chi_n(\tnorm{\psi}_{H^{1/2}}) K_{\eps,n}(\psi).
    \end{align}
    This vector field~$K_{\eps,\delta}$ is locally Lipschitz since the sum is locally finite and each summand is locally Lipschitz. 
    It remains to verify (K1) and (K2). 

    For (K1): if~$\tnorm{\psi}_{H^{1/2}}\leq R < \lfloor R \rfloor+1$, then~$\chi_n(\tnorm{\psi}_{H^{1/2}})=0$ unless~$n\leq \lfloor R\rfloor+2$. Thus~$K_{\eps,\delta}(\mathscr{B}(R))$ is bounded because only finitely many~$y_{k,j}$'s enter~$\mathscr{B}(R)$, and is contained in the subspace
    \begin{align}
        \Span\braces{ y_{k,j}\mid  1\leq j\leq N_k,\; 1\leq k\leq \lfloor R\rfloor +2}
    \end{align}
    which is finite-dimensional.
    In particular,~$K_{\eps,\delta}$ is compact on bounded sets.

    For (K2): for any~$\psi\in H^{\frac{1}{2}}$, if~$\chi_n(\tnorm{\psi})\neq 0$, then~$\tnorm{\psi}_{H^{1/2}}\leq n+1 <n+2$ and hence
    \begin{align}
        \tnorm{ K_{\eps,n}(\psi)-K_{\eps}(\psi)}_{H^{1/2}}<\delta/2.
    \end{align}
    It follows that 
    \begin{align}
        \tnorm{K_{\eps,\delta}(\psi)-K_{\eps}(\psi)}_{H^{1/2}}
        \leq \sum_{n} \chi_n(\tnorm{\psi}_{H^{1/2}})\tnorm{ K_{\eps,n}(\psi)-K_\eps(\psi)}_{H^{1/2}} <\delta.
    \end{align}
\end{proof}

Later we will use the vector field~$\psi\mapsto P^+\psi-P^-\psi -K_{\eps,\delta}(\psi)$ as a replacement of~$\nabla J_\eps$, to generate a suitable deformation.
Before that we consider some general deformations.

\

A vector field~$W$ on~$H^{\frac{1}{2}}(\mathbb{T}^3,\C^4)$ is called an \emph{admissible pseudo-gradient field} if it has the form 
\begin{align}
    W(\psi)= \eta(\psi) \parenthesis{ (P^+ -P^- )\psi - \widetilde{K}(\psi)},
\end{align}
where
\begin{itemize}
    \item[(W1)] $\eta\colon H^{\frac{1}{2}}\to [0,1]$ is a locally Lipschitz function,
    \item[(W2)] $\widetilde{K}\colon H^{\frac{1}{2}}\to H^{\frac{1}{2}}$ is locally Lipschitz and is compact on bounded sets;
    \item[(W3)] $W|_{\p\mathscr{C}_1(R)} = 0$;
    \item[(W4)] On the region~$\braces{\psi\colon \eta(\psi)>0}$, there holds~$\dd J_\eps(\psi)[P^+\psi -P^- \psi -\widetilde{K}(\psi)]\geq 0$;
    \item[(W5)] $\tnorm{W(\psi)}_{H^{1/2}} \leq 1$ for any~$\psi\in H^{\frac{1}{2}}$. 
\end{itemize}
This notion is slightly different from the classical one as in~\cite{Ambrosetti2007Nonlinear}. 
In particular, we require it to be uniformly bounded but not necessarily comparable to the genuine gradient. 

Being bounded and locally Lipschitz, the vector field~$-W$ generates a global flow~$\phi^W_t\colon H^{\frac{1}{2}}\to H^{\frac{1}{2}}$: 
\begin{align}
    \begin{cases}
        \frac{\p\phi^W_t(\psi)}{\p t} = - W(\phi^W_t(\psi)), & \quad \forall \psi \in H^{\frac{1}{2}}, \; \forall t\geq 0, \\
        \phi^W_0(\psi)=\psi, & \quad \forall \psi\in H^{\frac{1}{2}}.
    \end{cases}
\end{align}
The property (W3) guarantees that~$\p\mathscr{C}_1(R)$ is fixed along the flow, and (W4) implies that the functional~$J_\eps$ does not increase along the flow. 
For this reason we call this flow~$(\phi^W_t)$ an admissible pseudo-gradient flow.
For each~$t\geq 0$, we call~$\phi^W_t$ a time-$t$ map of the admissible pseudo-gradient flow. 

Note that (W5) implies that 
\begin{align}
    \tnorm{\phi^W_t(\psi)-\psi}_{H^{1/2}}\leq \int_0^{|t|} \tnorm{W(\phi^W_s(\psi))}_{H^{1/2}}\dd{s} \leq |t|.
\end{align}
Thus, there is no blowup in finite time, and bounded sets are flowed into bounded sets in finite time. 

Fix a~$\psi\in H^{\frac{1}{2}}$ and consider a single flow line~$\phi^W_t(\psi)$.
Taking the negative projection of the flow equation, we have  
\begin{align}
    P^-\parenthesis{\frac{\p \phi^W_t(\psi)}{\p t}} 
    = \eta(\phi^W_t(\psi))P^-\phi^W_t(\psi) 
    + \eta(\phi^W_t(\psi))P^-\widetilde{K}(\phi^W_t(\psi)).
\end{align}
It follows that 
\begin{align}
    P^-\phi^W_t(\psi)
    = e^{u^W(t,\psi)}P^-\psi + C^W(t,\psi)
\end{align}
with 
\begin{align}
    u^W(t,\psi)=\int_0^t \eta(\phi^W_s(\psi))\dd s, \qquad  0\leq u^W(t,\psi)\leq t ,
\end{align}
\begin{align}
    C^W(t,\psi)
    = \int_0^t e^{\int_s^t \eta(\phi^W_\tau(\psi))\dd \tau } \eta(\phi^W_s(\psi))P^-\widetilde{K}(\phi^W_s(\psi))\dd{s}. 
\end{align}
We see that for any bounded set~$B\subset H^{\frac{1}{2}}$ and any bounded time interval~$[0,T]$, the map 
\begin{align}
    C^W\colon [0,T]\times B \to H^{\frac{1}{2},-}
\end{align}
is continuous and compact. 

\

Let~$\Gamma$ be the collection of finite compositions of time maps of admissible pseudo-gradient flows:
\begin{align}
    \Gamma\coloneqq \braces{ \phi^{W_k}_{t_k}\circ \cdots\circ \phi^{W_1}_{t_1} \mid W_j \mbox{ is admissible }, t_j\geq 0, \; k\in\mathbb{N} }. 
\end{align}
Putting~$t_j=0$ we see that~$\id\in\Gamma\neq \emptyset$.
Moreover,~$\Gamma$ is closed under composition (so it is a semigroup), and any element~$\tilde{\phi}$ is homotopic to the identity map, say via~$\Phi(t)$:~$\Phi(0)=\id$,~$\Phi(1)=\tilde{\phi}$, and for~$t\in [0,1]$,
\begin{itemize}
    \item[(1)] $\Phi(t)$ fixes~$\p\mathscr{C}_1(R)$ pointwise; 
    \item[(2)] $P^-\Phi(t,\psi)=e^{u(t,\psi)}P^-\psi+ C_\Phi(t,\psi)$ where~$u$ is continuous and bounded, and~$C_\Phi$ is continuous and compact.
\end{itemize}

\begin{lemma}\label{lemma:intersection}
    For any~$\tilde{\phi}\in \Gamma$ and~$\Phi(t)$ as above,
    \begin{align}
        \Phi(t,\mathscr{C}_1(R))\cap \mathscr{S}^+(r)\neq \emptyset,\qquad \forall t\in [0,1].
    \end{align}
\end{lemma}

\begin{proof}
    This can be seen via a modification of the argument in~\cite{EstebanSere1995stationary}. 
    We give a proof for completeness. 
    
    For each~$n\geq 1$, let~$H^{\frac{1}{2},-}_n$ be the $n$-dimensional subspace of~$H^{\frac{1}{2},-}$ spanned by the eigenspinors corresponding to the largest~$n$ negative eigenvalues of~$\pD$, and let~$P^-_n$ denote the orthogonal projection onto this space. 
    Let
    \begin{align}
        \mathscr{C}_{1,n}(R)\coloneqq \mathscr{C}_1(R)\cap \parenthesis{  H^{\frac{1}{2},-}_n \oplus \R e}.
    \end{align}
    It is a cylinder in a finite-dimensional space. 
    For each~$t\in[0,1]$, consider the map 
    \begin{align}
        \widehat{\Phi}^t_{n} \colon \mathscr{C}_{1,n}(R)\to H^{\frac{1}{2},-}_n\times \R, \qquad 
        \widehat{\Phi}^t_n(\psi) \coloneqq  \parenthesis{  P_n^- \Phi(t,\psi), \tnorm{P^+\Phi(t,\psi)}_{H^{1/2}} } 
    \end{align}
    For each~$\psi=\psi^-+ \lambda e\in\p \mathscr{C}_{1,n}(R)$, we have~$\Phi(t,\psi)=\psi$ and hence~$\widehat{\Phi}^t_n (\psi)=(\psi^-, \lambda)$ which is different from~$(0,r)$, while for~$t=0$ there is precisely one element,~$re$, mapped to~$(0,r)$. 
    By the homotopy invariance of the finite-dimensional Brouwer degree, we see that for each~$t\in [0,1]$, 
    \begin{align}
        \deg(\widehat{\Phi}^t_{n}, \mbox{ int } \mathscr{C}_{1,n}(R); (0,r))
        =\deg(\widehat{\Phi}^0_{n}, \mbox{ int }\mathscr{C}_{1,n}(R); (0,r))=1 \neq 0.
    \end{align}
    Thus there exist~$\psi_n=\psi_n^- + \lambda_n e \in \mathscr{C}_{1,n}(R)$ such that~$\widehat{\Phi}^t_n(\psi_n)=(0,r)$, that is,
    \begin{align}
        P^-_n\Phi(t,\psi_n)=0, & & 
        \tnorm{P^+\Phi(t,\psi_n)}_{H^{1/2}}=r. 
    \end{align}
    We claim that, passing to a subsequence if necessary,~$\psi_n\to \psi\in \mathscr{C}_1(R)$ and~$\Phi(t,\psi_n)\to \Phi(t,\psi)\in\mathscr{S}^+(r)$, which verifies the nonempty intersection. 
    Indeed, since~$u$ is bounded and~$C_\Phi$ is compact, while~$\tnorm{\psi_n}_{H^{1/2}}\leq \sqrt{2} R$, by the structural assumption of~$\Phi(t,\psi)$,
    \begin{align}
        P^-\Phi(t,\psi_n) = e^{u(t,\psi_n)}\psi_n^- + C_\Phi(t,\psi_n), 
    \end{align}
    we see that, passing to a subsequence if necessary,~$\lambda_n\to \lambda\in [0,R]$,~$u(t,\psi_n)\to u_\infty\in \R$ and
    \begin{align}
        \psi_n^- = -e^{- u(t,\psi_n)}P_n^- C_\Phi(t,\psi_n) \to \psi^- \in H^{\frac{1}{2},-}
    \end{align}
    with~$\tnorm{\psi^-}_{H^{1/2}}\leq R$. Hence~$\psi_n=\psi_n^-+\lambda_n e\to \psi^-+\lambda e\in \mathscr{C}_1(R)$.
    Moreover,~$\tnorm{P^+\Phi(t,\psi)}_{H^{1/2}}=r$ by continuity. 
    
    To see that~$P^-\Phi(t,\psi)=0$, we note that for each~$n\geq 1$, we have~$P_n^-\Phi(t,\psi_n)=0$; and if~$n\geq k$, then~$P_k^- P_n^- = P_k^-$. 
    Thus for~$n\geq k$,
    \begin{align}
        P_k^- \Phi(t,\psi_n)=P_k^- P_n^-\Phi(t,\psi_n) =0, 
    \end{align}
    Since~$\psi_n\to\psi$ and~$\Phi(t,\cdot)$ is continuous, we see that
    \begin{align}
        P_k^-\Phi(t,\psi)=0
    \end{align}
    for any~$k\geq 1$. 
    Letting~$k\to+\infty$, using the density of~$\cup_{n\geq1} H^{\frac{1}{2},-}_n$ in~$H^{\frac{1}{2},-}$, we conclude that~$P^-\Phi(t,\psi)=0$. 
    Therefore,~$\Phi(t,\psi)\in \mathscr{S}^+(r)\cap \Phi(\mathscr{C}_1(R))$, as desired. 
\end{proof}

With these preparations, we can define the min-max level
\begin{align}
    \Lambda_1(\eps)\coloneqq \inf_{\tilde{\phi}\in\Gamma} \; \sup_{\tilde{\phi}(\mathscr{C}_1(R))} J_\eps 
\end{align}
Since~$\id\in\Gamma$, we have 
\begin{align}
    \Lambda_1(\eps)\leq \sup_{\mathscr{C}_1(R)} J_\eps \leq \frac{m-a}{2m}R^2. 
\end{align}
On the other hand, for any~$\tilde{\phi}\in\Gamma$, using Lemma~\ref{lemma:intersection} (applied with~$t=1$ and~$\Phi(1)=\tilde{\phi}$) and Lemma~\ref{lemma:positivity in interior-1}, we see that
\begin{align}
    \sup_{\tilde{\phi}(\mathscr{C}_1(R))} J_\eps\geq \inf_{\tilde{\phi}(\mathscr{C}_1(R)) \cap \mathscr{S}^+(r)} J_\eps\geq  \inf_{\mathscr{S}^+(r)} J_\eps \geq C_*>0.
\end{align}
Taking infimum over~$\tilde{\phi}\in\Gamma$ leads to~$\Lambda_1(\eps)\geq C_*>0$. 
We emphasize again that the lower bound~$C_*$ and the upper bound~$\frac{m-a}{2m}R^2$ are independent of~$\eps\in (0,1]$ but depend on~$m$ and~$a$.
In particular, if~$a$ is close to~$m$, then both bounds are small. 

\ 

It is then standard to show that~$\Lambda_1(\eps)$ is a critical level for~$J_\eps$. 
Indeed, it suffices to show that there is a sequence~$(\psi_n)$ with 
\begin{align}
    J_\eps(\psi_n)\to \Lambda_1(\eps), &&  \mbox{ and } & & 
    \dd J_\eps(\psi_n)\to 0 \quad \mbox{ in } H^{-\frac{1}{2}},
\end{align}
as~$n\to +\infty$. 
We have seen that for~$\eps\in(0,1]$, the functional~$J_\eps$ satisfies the Palais--Smale condition at any level. 
Thus~$\psi_n$ subconverges to a critical point~$\psi_\infty$ at the level of~$\Lambda_1(\eps)$, as desired. 

Suppose there is no such sequence. 
Then there exist~$\delta >0$ and~$\rho>0$ such that
\begin{align}
    |J_\eps(\psi)-\Lambda_1(\eps)|\leq 2\rho \Longrightarrow \tnorm{\dd J_\eps(\psi)}_{H^{-1/2}} \geq 2\delta.
\end{align}
We may assume that~$\rho<\frac{C_*}{4}$. 
For such a~$\delta>0$, let~$K_{\eps,\delta}$ be the vector field given in Lemma~\ref{lemma:global compact approx}. 
Also take a Lipschitz function~$\eta\colon \R\to [0,1] $ such that 
\begin{align}
    \eta\equiv 1 \quad \mbox{ on } \; [\Lambda_1(\eps)-\rho,\Lambda_1(\eps)+\rho]
\end{align}
and 
\begin{align}
    \eta\equiv 0 \quad \mbox{ in } \; \R\setminus [\Lambda_1(\eps)-2\rho, \Lambda_1(\eps)+2\rho].
\end{align}
Then consider the vector field 
\begin{align}
    W_{\eps,\delta}(\psi)\coloneqq \eta(J_\eps(\psi)) \frac{ P^+\psi -P^-\psi-K_{\eps,\delta}(\psi)}{ 1+ \tnorm{ P^+\psi-P^-\psi - K_{\eps,\delta}(\psi)}}_{H^{1/2}}.
\end{align}
By construction,~$W_{\eps,\delta}$ satisfies (W1-3) and (W5), while for (W4): in the set~$\braces{\eta>0}\subset \braces{ |J_\eps(\psi)-\Lambda_1(\eps)|<2\rho }$ there holds
\begin{align}
    \dd J_\eps(\psi)[ P^+\psi-P^-\psi - K_{\eps,\delta}(\psi)]
    =&\dd J_\eps(\psi) [\nabla J_\eps(\psi) + K_\eps(\psi) -K_{\eps,\delta}(\psi) ] \\
    \geq & \tnorm{\nabla J_\eps(\psi)} _{H^{-1/2}}\parenthesis{ \tnorm{\nabla J_\eps(\psi)}_{H^{-1/2}} -\delta}\\
    \geq& 2\delta (2\delta-\delta)=2\delta^2>0.
\end{align}
Thus~$W_{\eps,\delta}$ is an admissible pseudo-gradient field. 
Let~$\phi^{W_{\eps,\delta}}_t$ be the flow associated to~$-W_{\eps,\delta}$.
Note that if~$|J_\eps(\phi^{W_{\eps,\delta}}_t(\psi))-\Lambda_1(\eps)|<\rho$, then~$\tnorm{\dd J_\eps(\phi^{W_{\eps,\delta}}_t(\psi))}_{H^{-1/2}}\geq 2\delta$,~$\eta(J_{\eps}(\phi^{W_{\eps,\delta}}_t(\psi)))=1$ and hence
\begin{align}
    -\frac{\dd}{\dd t} J_\eps(\phi^{W_{\eps,\delta}}_t(\psi))
    =&  \dd J_\eps[W_{\eps,\delta}] \Big|_{ \phi^{W_{\eps,\delta}}_{t}(\psi) }
    = \frac{ \dd J_\eps \parenthesis{ \nabla J_\eps + K_\eps -K_{\eps,\delta}} }{ 1+ \tnorm{ \nabla J_\eps + K_\eps - K_{\eps,\delta} }_{H^{-1/2}} } \Big|_{\phi^{W_{\eps,\delta}}_t(\psi)} \\
    \geq& \frac{ \tnorm{\nabla J_\eps}^2_{H^{-1/2}} - \delta \tnorm{\nabla J_\eps}_{H^{-1/2}} }{1+\tnorm{\nabla J_\eps}_{H^{-1/2}}+\delta }\Big|_{\phi^{W_{\eps,\delta}}_t(\psi)} \\
    \geq& \frac{ \frac{1}{2}\tnorm{\nabla J_\eps}^2_{H^{-1/2}} }{ 1+ \frac{3}{2}\tnorm{\nabla J_\eps}_{H^{-1/2}} }\Big|_{\phi^{W_{\eps,\delta}}_t(\psi)}  \\
    \geq & \frac{\delta^2}{2+3\delta}>0.
\end{align}

Since~$\Lambda_1(\eps)$ is defined as an infimum, we can choose~$\tilde{\phi}\in\Gamma$ such that 
\begin{align}
    \sup_{\tilde{\phi}(\mathscr{C}_1(R))} J_\eps \leq \Lambda_1(\eps)+\rho.
\end{align}
Compose~$\tilde{\phi}$ with the time-$T$ map of the flow induced by~$-W_{\eps,\delta}$ as above, where~$T$ is so large that 
\begin{align}
    \frac{\delta^2}{2+3\delta} T > 2\rho.
\end{align}
Then~$\phi^{W_{\eps,\delta}}_T\circ \tilde{\phi}\in\Gamma$ and
\begin{align}
    \sup_{\phi^{W_{\eps,\delta}}_T\circ \tilde{\phi} (\mathscr{C}_1(R))} J_\eps \leq \Lambda_1(\eps)-\rho
\end{align}
which contradicts the definition of~$\Lambda_1(\eps)$. 
Therefore~$\Lambda_1(\eps)$ is a critical level of~$J_\eps$, namely, there exists~$\psi_\eps\neq 0$ with
\begin{align}
    J_\eps (\psi_\eps) = \Lambda_1(\eps)\in [C_*,\frac{m-a}{2m}R^2], & & 
    \dd J_\eps(\psi_\eps)=0. 
\end{align}

\begin{rmk}\label{rmk:multiplicity}
    It is tempting to hope for multiple solutions at this stage, especially when~$F$ is even in~$\psi$ or has more symmetries, as in~\cite{DingLiu2014periodic} and~\cite{DingRuf2008Solutions}. 
    This would require a restriction of the value of~$a$, or equivalently, that~$m\gg a$. 
    The latter condition on~$m$ is not too restrictive since, in the original setting, we should have~$mc^2$ where~$c$ stands for the speed of light. 
    There are plenty of works on the limiting behavior as~$c\to+\infty$. 
    We leave this, together with the multiplicity, to a future work. 
\end{rmk}


\section{Uniform estimates of the perturbed solutions}\label{sect:uniform estimate}

For each~$\eps\in (0,1]$, let~$0\neq \psi_\eps\in H^{\frac{1}{2}}(\mathbb{T}^3,\C^4)$ be a solution of~\eqref{eq:perturbative Dirac eq}. 
Moreover, suppose that 
\begin{align}\label{eq:uniform level}
    0<c_1 < J_\eps(\psi_\eps) < c_2<+\infty. 
\end{align}
We have seen that this is the case for those min-max solutions found in the previous sections. 
In this section we try to get some uniform estimates for these perturbed solutions, under additional assumptions. 

\begin{lemma}\label{lemma:estimate of F}
    There exists a constant~$C=C(\alpha_1, c_2)>0$ such that 
    \begin{align}
        0\leq \int_{\mathbb{T}^3} F_\eps(\psi_\eps)\, \dv\leq C.
    \end{align}
\end{lemma}

\begin{proof}
    Testing~\eqref{eq:perturbative Dirac eq} against~$\psi_\eps$, and using (F3) we get 
    \begin{align}
        2J_\eps(\psi_\eps) + 2\int_{\mathbb{T}^3} F_\eps(\psi_\eps)\dv 
        =&\int_{\mathbb{T}^3} \Abracket{\psi_\eps,\pD\psi_\eps}-a|\psi_\eps|^2\dv
        =\int_{\mathbb{T}^3} \p_\psi F_\eps(x,\psi_\eps)[\psi_\eps]\dv \\
        \geq &\alpha_1\int_{\mathbb{T}^3} F_\eps(\psi_\eps) \dv.
    \end{align}
    As~$\alpha_1>2$ and~$J_\eps(\psi_\eps)$ is bounded, it follows that 
    \begin{align}
        0\leq \int_{\mathbb{T}^3} F_\eps(\psi_\eps)\dv \leq \frac{2c_2}{\alpha_1-2}.
    \end{align}
\end{proof}

As a consequence, since~$F\geq 0$, we have 
\begin{align}\label{eq:uniform estimate for perturbation}
    0\leq \eps\int_{\mathbb{T}^3} |\psi_\eps|^{\alpha_2}\dv\leq C.  
\end{align}

\begin{rmk}
    From the proof of Lemma~\ref{lemma:estimate of F} we also see that 
\begin{align}
    0\leq \int_{\mathbb{T}^3} \p_\psi F_\eps(x,\psi_\eps)[\psi_\eps] \dv \leq 2c_2 + 2\frac{2c_2}{\alpha_1-2}= \frac{2 c_2\alpha_1}{\alpha_1-2}.
\end{align}
Moreover, due to (F4) and Lemma~\ref{lemma:estimate of F}, we see that 
\begin{align}
    \int_{\mathbb{T}^3} m\bar{\psi_\eps}\psi_\eps - a|\psi_\eps|^2\dv 
    \leq & (m-a)\int_{\mathbb{T}^3} \bar{\psi_\eps}\psi_\eps\dv \\
    \leq & (m-a)\parenthesis{\int_{\mathbb{T}^3} |\bar{\psi_\eps}\psi_\eps|^\nu\dv}^{\frac{1}{\nu}} (\vol(\mathbb{T}^3))^{1-\frac{1}{\nu}} \\
    \leq & (m-a) (\vol(\mathbb{T}^3))^{1-\frac{1}{\nu}} \parenthesis{\frac{1}{A_3} (\int_{\mathbb{T}^3} F_\eps(\psi_\eps)+A_4\dv  )  }^{\frac{1}{\nu}}\\
    \leq & C<+\infty.
\end{align}
It follows that 
\begin{align}
    \int_{\mathbb{T}^3} \Abracket{\psi_\eps,\D\psi_\eps}\dv \geq -C >-\infty.
\end{align}
\end{rmk}

    \begin{rmk}
        If one has  
        \begin{align}\label{eq:a strong assumption}
            \int_{\mathbb{T}^3} \p_\psi F_\eps(x,\psi_\eps) [\gamma^0\psi_\eps] \dv\leq C
        \end{align}
        then testing~\eqref{eq:perturbative Dirac eq} against~$\gamma^0\psi_\eps$, using~$|\bar\psi_\eps\psi_\eps|\le|\psi_\eps|^2$ and the uniform estimate~$\eps\int_{\T^3}|\psi_\eps|^{\alpha_2}\le C$, gives 
        \begin{align}
            (m-a)\int_{\mathbb{T}^3} |\psi_\eps|^2\dv \leq C
        \end{align}        
        which is the desired uniform~$L^2$ estimate. 
        But~\eqref{eq:a strong assumption} seems incompatible with the Soler type nonlinearity. 
        This is why we pick the assumption (F5) which was also used in~\cite{EstebanSere1995stationary}.        
    \end{rmk}

In the following we try to see whether the norms~$\|\psi_\eps\|_{L^3}$ are uniformly bounded or not. 
Suppose this is not the case. Then along a sequence~$\eps_n\searrow 0$, we have 
\begin{align}
    \lambda_n\equiv \|\psi_{\eps_n}\|_{L^3(\mathbb{T}^3)}\to +\infty.
\end{align}
The normalized spinors 
\begin{align}
    \varphi_n\coloneqq \frac{\psi_{\eps_n}}{\lambda_n}
\end{align}
satisfy~$\|\varphi_n\|_{L^3}=1$, and the equations
\begin{align}\label{eq:normalized eqn}
    \pD\varphi_n-a\varphi_n=\frac{\p_\psi F(x,\psi_{\eps_n})}{\lambda_n} + \alpha_2\eps_n |\psi_{\eps_n}|^{\alpha_2 -2 }\varphi_n, \qquad \mbox{ on } \; \mathbb{T}^3.
\end{align}
Observe that 
\begin{align}
    \|\alpha_2\eps_n |\psi_{\eps_n}|^{\alpha_2 -2 }\varphi_n \|_{L^{\frac{\alpha_2}{\alpha_2 -1}}}
    = & \frac{\alpha_2 \eps_n^{\frac{1}{\alpha_2}} }{\lambda_n}
     \parenthesis{ \int_{\mathbb{T}^3} \left|\eps_n^{\frac{\alpha_2-1}{\alpha_2}} |\psi_{\eps_n}|^{\alpha_2 -1} \right|^{\frac{\alpha_2}{\alpha_2-1}} \dv   }^{\frac{\alpha_2-1}{\alpha_2}}\\
    \leq& \frac{\alpha_2 \eps_n^{\frac{1}{\alpha_2}} }{\lambda_n}
    \parenthesis{ \int_{\mathbb{T}^3} \eps_n |\psi_{\eps_n}|^{\alpha_2}\dv}^{\frac{\alpha_2-1}{\alpha_2}} \\ 
    \leq& \frac{\alpha_2 \eps_n^{\frac{1}{\alpha_2}} }{\lambda_n}\parenthesis{ \frac{2c_2}{\alpha_1-2} }^{\frac{\alpha_2-1}{\alpha_2}}
\end{align}
thanks to~\eqref{eq:uniform estimate for perturbation}.
Since~$\alpha_2\in (2,3)$, we have~$\frac{\alpha_2}{\alpha_2-1}>\frac{3}{2}$. 
Thus the perturbation term will disappear in the limit as~$\eps_n\to 0$.

For the~$\p_\psi F$ term, it is more complicated. 
By (F5) we have 
\begin{align}
    \left|\frac{1}{\lambda_n} \p_\psi F(x,\psi_{\eps_n})\right|
    \leq C \parenthesis{1+ F(x,\psi_{\eps_n})^{\frac{1}{\beta}}} |\varphi_n| \in L^{\frac{3\beta}{3+\beta}}
\end{align}
and 
\begin{align}
    \|\frac{1}{\lambda_n}\p_\psi F(x,\psi_{\eps_n}(x))\|_{L^{\frac{3\beta}{3+\beta}}} 
    \leq C \parenthesis{ \vol(\mathbb{T}^3)+\frac{2c_2}{\alpha_1-2} }^{\frac{1}{\beta}}.
\end{align}
due to Lemma~\ref{lemma:estimate of F}.
Note that~$\frac{3\beta}{3+\beta}>\frac{3}{2}$ since~$\beta>3$ by assumption.
Therefore,~$ (\pD-a)\varphi_n \in L^p(\mathbb{T}^3)$ with
\begin{align}
    p= \min\braces{ \frac{3\beta}{3+\beta}, \frac{\alpha_2}{\alpha_2-1} } >\frac{3}{2},  
\end{align}
and 
\begin{align}
    \|(\pD-a)\varphi_n\|_{L^p} \leq 2C \parenthesis{ \vol(\mathbb{T}^3)+\frac{2c_2}{\alpha_1-2} }^{\frac{1}{\beta}}. 
\end{align}
Then we see that the sequence~$(\varphi_n)$ is uniformly bounded in~$W^{1,p}$ and, after passing to a subsequence, converges in~$L^3$ to some limit~$\varphi_\infty$ with~$\|\varphi_\infty\|_{L^3}=1$.

Moreover, combining (F4) and Lemma~\ref{lemma:estimate of F}, we see that 
\begin{align}
    \int_{\mathbb{T}^3} |\bar{\varphi}_{n}\varphi_{n}|^\nu\dv \leq \frac{1}{A_3\lambda_n^{2\nu}} \parenthesis{ \int_{\mathbb{T}^3} F_{\eps_n}(x,\psi_{\eps_n})\dv + A_4\vol(\mathbb{T}^3)}
    \leq \frac{C+A_4 \vol(\mathbb{T}^3)}{A_3\lambda_n^{2\nu}}\to 0
\end{align}
as~$n\to+\infty$. 
It follows that~$\bar{\varphi}_\infty\varphi_\infty\equiv 0$, i.e.~$\varphi_\infty$ is in the Lorentz null cone. 
In particular,~$\varphi_\infty$ cannot lie in~$\Eigen(\pD;m)$ since this eigenspace consists of constant spinors with only upper components, for which the Lorentz scalar coincides with the Euclidean amplitude. 
This can be strengthened as follows. 

Let~$\proj_m$ denote the orthogonal projection to~$\Eigen(\pD;m)$ and~$\proj_m^\perp = \id -\proj_m$ the complementary projection. 

\begin{lemma}
Let~$\nu\in (1,\frac{3}{2})$. 
    There exist~$\kappa_1>0$ and~$\kappa_2>0$ such that  
    \begin{align}
        \|\varphi\|_{L^3}=1,\quad \|\bar{\varphi} \varphi\|_{L^\nu}\leq\kappa_1 \quad\Longrightarrow\quad \|\proj_m^\perp \varphi\|_{L^3}\geq\kappa_2. 
    \end{align}
\end{lemma}

\begin{proof}
    If not, then there would be a sequence~$(\varphi_k)$ such that 
    \begin{align}
        \|\varphi_k\|_{L^3}=1, \quad \|\bar{\varphi}_k\varphi_k\|_{L^\nu} \to 0, 
        \quad \|\proj_m^\perp \varphi_k\|_{L^3}\to 0, \qquad \mbox{ as } \; k\to+\infty.
    \end{align}
    As~$\|\proj_m\varphi_k\|_{L^3}\leq C\|\varphi_k\|_{L^3}=C$, the sequence~$(\proj_m\varphi_k)$ is bounded in~$\Eigen(\pD;m)\cong\C^2$. 
    Thus, passing to a subsequence if necessary, we see that~$\proj_m\varphi_k$ converges to, say,~$\varphi_{*}$.
    Together with the convergence~$\|\proj_m^\perp \varphi_k\|_{L^3}\to0$, we see that 
    \begin{align}
        \|\varphi_{*}\|_{L^3}=1, & & \bar{\varphi}_* \varphi_*=0.
    \end{align}
    This is impossible, since~$\varphi_*\in\Eigen(\pD;m)$ and~$\bar{\varphi}_*\varphi_*=|\varphi_*|^2$.
\end{proof}

\begin{lemma}
    For any~$a\in [0,m]$, there exists a constant~$C=C(m,\mathbb{T}^3)$ such that 
    \begin{align}
        \|\proj_m^\perp \varphi\|_{L^3}\leq C\|\proj_m^\perp(\pD-a)\varphi\|_{L^{\frac{3}{2}}}. 
    \end{align}
\end{lemma}

    \begin{proof}
        We again use a contradiction argument. 
        So suppose there exist sequences~$(a_k)\subset [0,m]$ and~$(\varphi_k)\subset W^{1,\frac{3}{2}}$ such that 
        \begin{align}
            \|\proj_m^\perp\varphi_k\|_{L^3}=1, & & 
            \|\proj_m^\perp(\pD-a_k)\varphi_k\|_{L^{\frac{3}{2}}}\to 0.
        \end{align}
        Note that 
        \begin{align}
            \|\varphi\|_{W^{1,\frac{3}{2}}} 
            \leq& C \parenthesis{ \| \varphi\|_{L^{\frac{3}{2}}} + \|\pD\varphi\|_{L^{\frac{3}{2}}} } \\
            \leq& C \parenthesis{ \| \varphi\|_{L^{\frac{3}{2}}} + \|(\pD-a)\varphi\|_{L^{\frac{3}{2}}} + a\|\varphi\|_{L^{\frac{3}{2}}} } \\
            \leq& C \parenthesis{ (m+1)\| \varphi\|_{L^{\frac{3}{2}}} + \|(\pD-a)\varphi\|_{L^{\frac{3}{2}}} }. 
        \end{align}
        Since~$\proj_m^\perp$ commutes with~$(\pD-a_k)$, we have 
        \begin{align}
            \|\proj_m^\perp \varphi_k\|_{W^{1,\frac{3}{2}}}
            \leq C(m, \vol(\mathbb{T}^3))\parenthesis{ \| \proj_m^\perp \varphi_k \|_{L^3} + \|\proj_m^\perp (\pD-a_k)\varphi_k\|_{L^{\frac{3}{2}}} },  
        \end{align}
        so the sequence is uniformly bounded.
        Thus, passing to a subsequence if necessary, we may assume 
        \begin{align}
            a_k\to a_*\in [0,m], & & \mbox{ and } & & 
            \proj_m^\perp\varphi_k \to \varphi_*=\proj_m^\perp \varphi_*\quad  \mbox{ in } \;  L^{\frac{3}{2}}. 
        \end{align}
        Moreover, from
        \begin{align}
            \pD(\proj_m^\perp\varphi_k)= \proj_m^\perp (\pD-a_k)\varphi_k+ a_k \proj_m^\perp\varphi_k \xrightarrow{L^{\frac{3}{2}}} 0+ a_*\varphi_*
        \end{align}
        it follows that 
        \begin{align}
            \pD\varphi_*=a_*\varphi_*
        \end{align}
        first in the sense of distribution, then by elliptic regularity in the classical sense.
        Since~$\varphi_* =\proj_m^\perp\varphi_*\perp \Eigen(\pD;m)$ while~$\dist(a_*, \Spect(\pD)\setminus\braces{m})>0$, we must have~$\varphi_*=0$.
        Thus in turn we have 
        \begin{align}
            \|\proj_m^\perp \varphi_k\|_{W^{1,\frac{3}{2}}} \to 0
        \end{align}
        as~$k\to+\infty$, and by Sobolev embedding also~$\|\proj_m^\perp\varphi_k\|_{L^3}\to 0$, a contradiction to the initial assumption~$\|\proj_m^\perp\varphi_k\|_{L^3}=1$ for any~$k$.
    \end{proof}

\begin{rmk}
    In the above lemma, we can fix~$a\in (0,m)$ and get a conclusion of the same form. 
    But here we want to emphasize the uniform constant~$C(m,\mathbb{T}^3)$ which does not depend on~$a\in [0,m]$.
\end{rmk}

Recall that~$\varphi_n=\frac{\psi_{\eps_n}}{\lambda_n}$ satisfies~$\|\varphi_n\|_{L^3}=1$,~$\|\bar{\varphi}_n\varphi_n\|_{L^\nu}\to 0$ and
\begin{align}
    (\pD-a)\varphi_n= \frac{1}{\lambda_n}\p_\psi F(x,\psi_{\eps_n}) + \alpha_2\eps_n|\psi_{\eps_n}|^{\alpha_2-2}\varphi_n.
\end{align}
We may assume~$\|\bar{\varphi}_n\varphi_n\|_{L^\nu}\leq \kappa_1$, so that (using~$\|\varphi_n\|_{L^3}=1$) 
\begin{align}
    0<\kappa_2\leq \|\proj_m^\perp \varphi_n\|_{L^3} 
    \leq& C(m) \|\proj_m^\perp (\pD-a)\varphi_n\|_{L^{\frac{3}{2}}} \\
    \leq& C(m,\vol(\mathbb{T}^3),\beta,\alpha_2)\parenthesis{ \|\frac{1}{\lambda_n}\p_\psi F(x,\psi_{\eps_n})\|_{L^{\frac{3\beta}{3+\beta}} } + \alpha_2\|\eps_n|\psi_{\eps_n}|^{\alpha_2-2}\varphi_n \|_{L^{\frac{\alpha_2}{\alpha_2-1}} } } \\
    \leq & C(m,\vol(\mathbb{T}^3),\beta,\alpha_2)
    \parenthesis{ A_5\parenthesis{\delta+ C_\delta \parenthesis{ \int_{\mathbb{T}^3} F(x,\psi_{\eps_n})\dv}^{\frac{1}{\beta}} } + \frac{\eps_n^{\frac{1}{\alpha_2}}}{\lambda_n} } \\
    \leq &  C(m,\vol(\mathbb{T}^3),\beta,\alpha_2, A_5)\parenthesis{ \delta + C_\delta c_2^{\frac{1}{\beta} } + \frac{\eps_n^{\frac{1}{\alpha_2}}}{\lambda_n} }.
\end{align} 
Choose first~$\delta$ sufficiently small, and then~$c_2$ sufficiently small (which is possible if~$a$ is close to~$m$), and finally~$n$ sufficiently large, we can make the right hand side smaller than~$\kappa_2$, thus arriving at a contradiction.

Therefore, we have proven the following.
\begin{prop}
    There exists~$m_*\in (0,m)$ such that for any~$a\in (m_*,m)$, the family~$(\psi_{\eps})_{\eps\in(0,1]}$ obtained in the previous sections is uniformly bounded in~$L^3$.
\end{prop}

Once the uniform~$L^3$ bound is obtained, a standard bootstrap argument applied to the equation of~$\psi_\eps$ gives a uniform~$H^1$ bound.

\begin{cor}
     There exist~$m_*\in (0,m)$ and~$C>0$ such that for any~$a\in (m_*,m)$, the family~$(\psi_{\eps})_{\eps\in(0,1]}$ obtained in the previous sections satisfies 
    \begin{align}
        \|\psi_\eps\|_{H^1(\T^3)} \leq C.
    \end{align}
\end{cor}

\begin{proof}
    In the equation
    \begin{align}
        \pD\psi_\eps
        = a\psi_\eps + \p_\psi F(x,\psi_\eps) + \alpha_2 \eps|\psi_{\eps}|^{\alpha_2 -2}\psi_\eps
    \end{align}
    the right hand side satisfies
    \begin{align}
        |a\psi_\eps + \p_\psi F(x,\psi_\eps)+ \alpha_2\eps|\psi_\eps|^{\alpha_2-2}\psi_\eps | 
        \leq 
        C\parenthesis{ 1+ |\psi_\eps|^{\alpha_2 -1}}
    \end{align}
    which is uniformly bounded in~$L^{\frac{3}{\alpha_2 -1 }}$ with~$\frac{3}{\alpha_2 -1}>\frac{3}{2}$. 
    The elliptic regularity theory implies that the spinors~$\psi_\eps $ are uniformly bounded in~$W^{1,\frac{3}{\alpha_2-1}} \hookrightarrow L^{\frac{3}{\alpha_2-2}}$. 
    Note that~$\frac{3}{\alpha_2-2}>3$ since~$\alpha_2<3$. 
    Repeating this subcritical bootstrap finitely many times reaches an exponent sufficient for an~$L^2$ right hand side, hence a uniform~$H^1$-bound.
\end{proof}

\section{Proof of the main theorem}\label{sect:last}

\subsection{Proof of Theorem~\ref{thm:main}}
Now we see that for each~$\eps\in (0,1]$ the perturbed functional~$J_\eps$ admits a min-max critical point~$\psi_\eps \in H^{\frac{1}{2}}$ such that 
\begin{align}
    J_\eps(\psi_\eps) \in (c_1, c_2), & & 
    \|\psi_\eps\|_{H^1}\leq C<+\infty. 
\end{align}
Up to a subsequence~${\eps_n\to 0^+}$, we assume that~$\psi_{\eps_n}\to \psi_0$ weakly in~$H^1$ and strongly in~$H^{\frac{1}{2}}\cap L^q$ for any~$1<q<6$.
In particular
\begin{align}
    \|\eps_n|\psi_{\eps_n}|^{\alpha_2 -2}\psi_{\eps_n} \|_{L^2}
    =\eps_n \|\psi_{\eps_n}\|_{L^{2(\alpha_2-1)}}^{\alpha_2 -1} 
    \leq C\eps_n \|\psi_{\eps_n}\|_{H^1}^{\alpha_2 -1} \to 0,\quad \mbox{ as } \eps_n \to 0^+
\end{align} 
since~$2(\alpha_2 -1)<6$. 
It follows that the limit spinor~$\psi_0$ solves the equation~\eqref{eq:NDE-spatial} in~$H^{-\frac{1}{2}}$, and thanks to the uniform~$H^1$ bound on these perturbed solutions we have 
\begin{align}
    \eps_n\int_{\mathbb{T}^3} |\psi_{\eps_n}|^{\alpha_2}\dv\to 0
\end{align}
as~$\eps_n\to 0^+$. 
Hence the critical levels also converge:
\begin{align}
    J(\psi_0)=\lim_{\eps_n\to 0^+} J_{\eps_n}(\psi_{\eps_n}) \geq c_1 >0. 
\end{align}
In particular,~$\psi_0\neq 0$. 

For the regularity of~$\psi_0$, note that~$\psi_0\in H^1\hookrightarrow L^6$ and thus 
\begin{align}
    \p_\psi F(x,\psi_0) \in L^{\frac{6}{\alpha_2-1}}
\end{align}
where~$\frac{6}{\alpha_2-1}>3$. 
Thus~$\psi_0\in W^{1,\frac{6}{\alpha_2 -1}} \hookrightarrow C^{\sigma}$ for some~$\sigma\in (0,1)$. 
By (F2),~$F\in C_{loc}^{1,\alpha}$ so~$\p_\psi F\in C_{loc}^\alpha$, thus the Schauder theory implies that~$\psi_0\in C_{loc}^{1,\sigma'}$ for some~$\sigma'\in (0,1)$.
Since~$\mathbb{T}^3$ is compact, by Schauder theory we conclude that~$\psi_0\in C^{1,\sigma''}$ for some~$\sigma''\leq \sigma'$.

\subsection{Proof of Proposition~\ref{prop:external}}

The argument is almost parallel to the previous one, and we only outline the modifications. 
Set
\begin{align}
    V(x)\coloneqq a+M(x), &  & 
    b\coloneqq \min_{\T^3}V(x), & & 
    W(x)\coloneqq V(x)-b.
\end{align}
By assumption,~$m_*<b\leq V(x)<m$, hence~$0\leq W(x)\leq m-b$. 

Consider the perturbed functional
\begin{align}
    J_{\eps,V}(\psi)
    = \int_{\mathbb{T}^3}\frac{1}{2}\Abracket{\psi,\pD\psi}-\frac{1}{2}V(x)|\psi|^2-F_\eps(x,\psi)\dv.
\end{align}
The verification of the Palais--Smale condition and deformation arguments are unchanged, since multiplication by~$V$ is bounded and the additional term~$|\pD|^{-1}(V\psi)$ is still compact on~$H^{1/2}$ as before.

The linking argument also carries over. Indeed, on the positive sphere,
\begin{align}
   J_{\eps,V}(\psi^+) 
   \geq \frac{1}{2}\left(1-\frac{\max V}{m}\right)\tnorm{\psi^+}^2_{H^{1/2}}
        -C\big(\tnorm{\psi^+}^{\alpha_1}_{H^{1/2}}
        +\tnorm{\psi^+}^{\alpha_2}_{H^{1/2}}\big) 
\end{align}
which is positive for sufficiently small radius since~$\max V<m$. 
On the top of the cylinder we apply Lemma~\ref{lemma:local control} with~$\mu=b$ to get~$R=R(b)$ and define~$\mathscr{C}_1(R(b))$ accordingly, where~$R(b)$ can be taken uniformly bounded. 
Since \(V\geq b\), this gives~$J_{\eps,V}\leq0$. 
Moreover, using~$V\geq b$ and the~$L^2$-orthogonality of~$H^{1/2,-}$ and~$\R e$,
\begin{align}
    \Lambda_{1,V}(\eps) \leq \sup_{\mathscr C_1(R(b))}J_{\eps,V} \leq \frac{m-b}{2m}R(b)^2.
\end{align}
Thus the corresponding perturbed critical points \(\psi_\eps\) satisfy the same uniform positive lower bound and an upper energy bound \(c_2\) which tends to zero as \(b\nearrow m\).

It remains only to check the uniform~$L^3$ estimate. If it failed, let
\begin{align}
    \lambda_n=\|\psi_{\eps_n}\|_{L^3}\to\infty, & & 
    \varphi_n=\frac{\psi_{\eps_n}}{\lambda_n}.
\end{align}
Then~$\|\varphi_n\|_{L^3}=1$, and the normalized equation becomes
\begin{align}
    (\pD-b)\varphi_n 
    = W\varphi_n +\frac{\p_\psi F(x,\psi_{\eps_n})}{\lambda_n} +\alpha_2\eps_n|\psi_{\eps_n}|^{\alpha_2-2}\varphi_n.
\end{align}
Repeating the argument of the previous section and using the projection estimate for~$\pD-b$, we obtain
\begin{align}
    \kappa_2 \leq C\parenthesis{ \|W\varphi_n\|_{L^{3/2}} +\delta +C_\delta c_2^{1/\beta}+ \frac{\eps_n^{1/\alpha_2}}{\lambda_n} }.
\end{align}
Noting that
\begin{align}
    \|W\varphi_n\|_{L^{3/2}} \leq C\|W\|_{L^\infty}\|\varphi_n\|_{L^3} \leq C(m-b),
\end{align}
we arrive at
\begin{align}
    \kappa_2 
    \leq C\parenthesis{ m-b+\delta+C_\delta c_2^{1/\beta} +\frac{\eps_n^{1/\alpha_2}}{\lambda_n} }.
\end{align}
Choosing first~$\delta>0$ small, then~$m_*$ sufficiently close to~$m$, and finally~$n$ sufficiently large, gives a contradiction. 
Consequently,~$(\psi_\eps)$ is uniformly bounded in~$L^3$, hence also in~$H^1$ by the same bootstrap argument as before.

We may therefore let~$\eps\to0$. 
The compactness and regularity arguments used in the proof of the main theorem yield a nontrivial~$C^1$ solution of~\eqref{eq:NLD with external field}.

\

\textbf{Conflict of Interest} The authors have no conflicts to disclose. 

\

\textbf{Data Availability Statement} Data sharing is not applicable to this article as no datasets were generated or analyzed during the current study.

\bibliographystyle{plain}
\bibliography{nonlinearDiracEqn}

\end{document}